\documentclass[11pt, reqno]{amsart}
\usepackage{amsaddr}
\usepackage[english]{babel}
\usepackage{color}
\usepackage{amsmath, latexsym, amsthm, amsfonts,bm,amssymb,mathscinet} 
\usepackage{xifthen}
\usepackage[text={5.8in,8.3in},centering]{geometry}
\usepackage{epsfig}
\usepackage[colorlinks,linkcolor=blue, citecolor=blue,urlcolor=black]{hyperref}
\usepackage[dvipsnames]{xcolor}
\usepackage{natbib}
\usepackage{enumitem}
\usepackage{verbatim}
\usepackage{footmisc}
\usepackage{cleveref}
\usepackage[ruled, linesnumbered]{algorithm2e}
\usepackage{algpseudocode}
\usepackage{tikz}
\usepackage{makecell}
\usepackage{subcaption}

\crefname{algocf}{algorithm}{algorithm}
\Crefname{algocf}{Algorithm}{Algorithms}

\numberwithin{equation}{section}

\usepackage{parskip}
\newcommand{\dd}{\,\mathrm{d}}
\renewcommand{\scr}[1]{{\mathfrak #1}}
\newcommand{\bb}[1]{{\mathbb #1}}

\newcommand{\OM}{{ \mathrm{O}M }}
\newcommand{\FM}{{ \mathrm{F}M }}
\newcommand{\Bm}{\begin{pmatrix}}
	\newcommand{\Em}{\end{pmatrix}}
\newcommand{\T}{{\prime}}
\newcommand{\der}[2]{ \ifthenelse{\isempty{#1}}{\frac{\dd}{\dd #2}}{\frac{\dd #1}{\dd #2}} }
\newcommand{\pder}[2]{ \ifthenelse{\isempty{#1}}{\frac{\partial}{\partial #2}}{\frac{\partial #1}{\partial #2}} }
\newcommand{\ind}{\mathbf{1}}

\newcommand{\eps}{\varepsilon}
\newcommand{\dist}{{ \mathrm{dist} }}
\theoremstyle{definition}
\newtheorem{thm}{Theorem}[section]
\newtheorem{prop}[thm]{Proposition}
\newtheorem{lem}[thm]{Lemma}
\newtheorem{cor}[thm]{Corollary}

\newtheorem{ex}[thm]{Example}

\newcommand{\abs}[1]{\ensuremath{\left\lvert #1 \right\rvert}}
\newcommand{\inner}[2]{\ensuremath{\left\langle #1 , #2 \right\rangle }}

\newcommand{\PP}{\mathbb{P}}

\usepackage{orcidlink}
\title{%
 Simulating conditioned diffusions on manifolds}
\date{\today}
\author{Marc Corstanje$^1$\orcidlink{0000-0002-2900-8773} \quad Frank van der Meulen$^1$\orcidlink{0000-0001-7246-8612} \quad Moritz Schauer$^2$\orcidlink{0000-0003-3310-7915} \quad Stefan Sommer$^3$\orcidlink{0000-0001-6784-0328}}
\address{%
	$^1$Department of Mathematics, Vrije Universiteit Amsterdam, NL\\
    $^2$Department of Mathematical Sciences, University of Gothenburg and Chalmers Technical University, Gothenburg, SE
	\\$^3$University of Copenhagen, DK
	}

\begin{document}
	\maketitle
	\let\thefootnote\relax\footnotetext{Contact: \href{mailto:M.A.Corstanje@vu.nl}{M.A.Corstanje@vu.nl}}

	\begin{abstract}
		To date, most methods for simulating  conditioned diffusions are limited to the Euclidean setting. The conditioned process can be constructed using a change of measure known as Doob's $h$-transform. The specific type of conditioning depends on a function $h$ which is typically unknown in closed form. To resolve this, we extend the notion of guided processes to a manifold $M$, where one replaces $h$ by a function based on the heat kernel on $M$. We consider the case of a Brownian motion with drift, constructed using the frame bundle of $M$, conditioned to hit a point $x_T$ at time $T$. We prove equivalence of the laws of the conditioned process and the guided process with a tractable Radon-Nikodym derivative. Subsequently, we show how one can obtain guided processes on any manifold $N$ that is diffeomorphic to $M$ without assuming knowledge of the heat kernel on $N$. We illustrate our results with numerical simulations of guided processes and  Bayesian parameter estimation based on discrete-time observations. For this, we consider both the torus and the Poincar\'e disk.    
    \\


	\end{abstract}

	\section{Introduction}
	
			 Let $M$ be a compact $d$-dimensional Riemannian manifold and let $V$ be a smooth vector field on $M$. Denote by $\tilde{V}$ the horizontal lift of $V$ to the frame bundle $\FM$ of $M$ and consider the Stratonovich stochastic differential equation (SDE) 
		 \begin{equation}
		 	\label{eq:SDE-UV}
		 	\dd U_t = \tilde{V}(U_t) \dd t +  H_i(U_t)\circ \dd W_t^i, \qquad U_0=u_0.
		 \end{equation}
		  where $H_1,\dots, H_d$ are the canonical horizontal vector fields, $u_0$ an orthonormal frame, and $W$ an $\bb{R}^d$-valued Brownian motion. We  use Einstein's summation convention to sum over all indices that appear in both sub- and superscript. The aim of this paper is to construct a method for simulation of the process $X = \pi(U)$, started at $x_0=\pi(u_0)$ and  conditioned on hitting some $x_T\in M$ at time $T>0$, where $\pi$ denotes the projection from $\FM$ to  $M$. Such a process is often called a pinned-down diffusion or diffusion bridge. For readers unfamiliar with SDEs on manifolds, we provide background and references in Section \ref{sec:SDEmanifolds}. 

\subsection{Motivation and applications}
    A diffusion defined by a stochastic differential equation  provides a flexible way to model continuous-time stochastic processes.  Whereas the process is described in continuous-time, observations are virtually always discrete in time, possibly irregularly spaced.  Transition densities are rarely known in closed form, complicating likelihood-based inference for parameters in either the drift- or diffusion coefficients of the SDE. To resolve this, data-augmentation algorithms have been proposed where the latent process in between observation time is imputed.  For diffusions that take values in a linear space, numerous works over the past two decades have shown how to deal with this problem (e.g.\ \cite{BeskosPapaspiliopoulosRobertsFearnhead, GolightlyWilkinsonChapter, PapaRobertsStramer, mider2021continuous}).  The corresponding problem where the process takes values in a manifold has received considerably less attention, despite its importance in various application settings. Examples include evolving particles on a sphere, evolution of positive-definite matrices and evolution of compositional data. From a more holistic point of view, our work potentially comprises an important step for statistical inference in   state space models on manifolds (\cite{manton2013primer}, Section III). 
     
    Applications of conditioned manifold processes also arise in shape analysis e.g.  \cite{sommerBridgeSimulationMetric2017,arnaudonDiffusionBridgesStochastic2022} and more generally in geometric statistics, e.g. \cite{jensen2022discrete},  which focuses on Lie groups and homogeneous spaces. See \cite{sommerProbabilisticApproachesGeometric2020} for a general overview. Bridges can be applied to model shape changes caused by progressing diseases as observed in medical imaging, or to model changes of animal morphometry during evolution and subsequently used for analyzing phylogenetic relationships between species.
    
\subsection{Related work}
    Bridge simulation on Euclidean spaces has been treated extensively in the literature. See for instance \cite{durham2002, Stuart, BeskosPapaspiliopoulosRoberts, beskos-mcmc-methods, hairer2009, LinChenMykland, bladt2016, whitaker2017, bierkens2020piecewise} and references therein. It is known that the conditioned process satisfies a stochastic differential equation itself, where a guiding term that depends on the transition density of the process is superimposed on the drift of the unconditioned process. Unfortunately, transition densities are only explicit in very special cases.  One class of methods to construct conditioned processes consists of approximating the true guiding term by a tractable substitute. The resulting processes are called guided proposals, guided processes or twisted processes. By explicit computation of the likelihood ratio between the true bridge and the approximating guided process any discrepancy can be  corrected for in (sequential) importance sampling or Markov chain Monte Carlo methods. The approach taken in this paper falls into this class. 
    
    Early instances of this approach are \cite{clark1990} and   \cite{delyon2006} where the guiding term  matches the drift of a Brownian bridge.   A different, though related,  approach is construct a guiding term based on $\nabla_x \log \tilde{p}(t,x; T, x_T)$, where $\tilde{p}$ is the transition density of a diffusion for which it is known in closed form (see \cite{schauer_guidedproposals2017} and \cite{bierkens2020simulation}). A linear process is a prime example of such diffusions.  
    
    There is extensive literature on manifold stochastic processes, their construction and bridges, see e.g. \cite{emeryStochasticCalculusManifolds1989,Hsu2002}. The construction of conditioned manifold processes is the subject of \cite{baudoin2004conditioning}. Direct adaptation of bridge simulation methods devised for Euclidean space by embedding the manifold in $\mathbb{R}^d$ easily run into difficulties as the diffusivity matrix becomes singular, the drift unbounded and the numerical scheme has to be particularly adapted to ensure that the process in the embedding space stays on the manifold. 
     In the geometric setting, the approach of \cite{delyon2006}  is generalized to manifolds in \cite{jensen2019simulation, jensen2021simulation,bui2023inference}. The drift here comes from the gradient of the square distance to the target, mimicking the drift of the Euclidean drift approximation. An important complication with this approach is handling the discontinuity of the gradient at the cut locus of the target point. This is introducing unnecessary complexities because the true guiding term is often smooth, and the scheme effectively approximates the smooth true transition density with a discontinuous guiding term which is harder to control, both numerically and analytically. With the approach in the presented paper, we exactly achieve smooth guiding terms in accordance with the true transition density and we avoid the cut locus issues. In addition, to ensure that the guided process in fact hits the conditioning point, \cite{jensen2019simulation, jensen2021simulation} impose a Lyapunov condition on the radial process. Verifying this assumption for a specified vector field $V$ is difficult. The approach taken in this paper avoids this condition altogether. 
    \cite{bui2023inference} don't formally prove absolute continuity though carefully  assess validity of their approach by extensive numerical simulation.  

\subsection{Approach and main results}
    We extend the approach of simulating Euclidean diffusion bridges from \cite{schauer_guidedproposals2017} to manifolds. If $X$ is a Euclidean SDE with drift $b$ and diffusion coefficient $\sigma$, an SDE for the conditioned process $X^h$ (with law $\bb{P}^h$) is obtained by superimposing the SDE for $X$ with the term $\sigma\sigma^\T \nabla\log h \dd t$, where $h$ is the intractable transition density of $X$. It is thus impossible to directly simulate $X^h$. Instead, $X$ is approximated by an "auxiliary" linear process $\tilde X$ with known closed-form transition density $g$, and the SDE for $X$ is superimposed with a term $\sigma\sigma^\T \nabla \log g \dd t$ resulting in a guided process $X^g$ (with law $\bb{P}^g$). This process is not equal to $X^h$ in law, but the likelihood ratio is computable and the resulting scheme allows expectations $\mathbb E f(X^h)$ over the conditioned process to be evaluated. As sketched in \Cref{fig:tikzdiagram-manifolds}, the approach presented in this paper is a geometric equivalent of this construction: We cannot simulate the conditioned process $X^h$ for a process $X$ on a manifold $N$ because of intractability of the transition density. Instead, if $N$ is diffeomorphic to a comparison manifold $M$ with closed form heat kernel $g$, we map this guiding term from $M$ to $N$ to give a guided process $X^g$ on $N$, again with computable likelihood ratio allowing evaluation of $\mathbb E f(X^h)$. The "auxiliary" process in this case is thus a Brownian motion on $N$. \Cref{fig:tikzdiagram-manifolds} sketches a comparison between this approach and \cite{schauer_guidedproposals2017}. 

   \begin{figure}[h]
	 \centering
		\begin{subfigure}{2.7in}
        \centering
		\begin{tikzpicture}
			\node[draw = black] at (-0.5,0) (a) {\makecell{$X$}};
			\node[draw = black] at (-0.5,-3) (b) {\makecell{$X^h$}};
            \node[draw = black] at (3, 0) (c) {\makecell{$\tilde{X}$}};
			\node[draw = black] at (3,-3) (d) {\makecell{$X^g$}};
			\draw [->] (a)--(b);
   			\path (a) -- node[midway, left] (text3) {$h$} (b);
      		\draw [->] (a)--(d);
            \draw [->] (c)--(d);
   			\path (c) -- node[midway, right] (text3) {$g$} (d);
            \draw [->] (d)--(b);
   			\path (d) -- node[midway, below] (text3) {$\dd\bb{P}^h / \dd\bb{P}^g$} (b);
		\end{tikzpicture}
		\caption{Euclidean: $X$ is conditioned using $h$ to obtain $X^h$. Since $h$ is intractable, we derive $g$ from an auxiliary process $\tilde{X}$, and use it to define $X^g$. The tractable likelihood ratio $\dd\bb{P}^{h}/\dd\bb{P}^{g}$ allows for evaluation of $\bb{E} f(X^h)$.}
		\end{subfigure} \hspace*{\fill}
		\begin{subfigure}{2.7in}
        \centering
		\begin{tikzpicture}
			\node[draw = black] at (-0.5,0) (a) {\makecell{$X$}};
			\node[draw = black] at (-0.5,-3) (b) {\makecell{$X^h$}};
            \node[draw = black] at (3, 0) (c) {\makecell{$\tilde{Y}$}};
			\node[draw = black] at (3,-3) (d) {\makecell{$X^g$}};
			\draw [->] (a)--(b);
   			\path (a) -- node[midway, left] (text3) {$h$} (b);
      		\draw [->] (a)--(d);
            \draw [->] (c)--(d);
   			\path (c) -- node[midway, right] (text3) {$g$} (d);
            \draw [->] (d)--(b);
   			\path (d) -- node[midway, below] (text3) {$\dd\bb{P}^h / \dd\bb{P}^g$} (b);
		\end{tikzpicture}
		\caption{Geometric: $X$ is conditioned using $h$ to obtain $X^h$. Since $h$ is intractable, we derive $g$ from an auxiliary process $\tilde{Y}$, a Brownian motion on a comparison manifold $M$, and use it to define $X^g$. The tractable likelihood ratio $\dd\bb{P}^{h}/\dd\bb{P}^{g}$ allows for evaluation of $\bb{E} f(X^h)$.}
		\end{subfigure}
  \caption{Diagrams sketching the approach for conditioning diffusion processes.}
  \label{fig:tikzdiagram-manifolds}
  \end{figure}
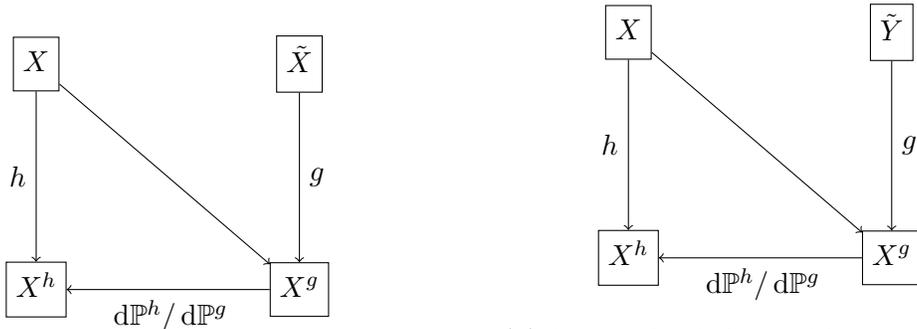

    Our approach is applicable for compact manifolds $N$ diffeomorphic to a manifold on which the heat kernel is known in closed form. This includes many cases such as Euclidean spaces with non-standard Riemannian metrics, spheres and ellipsoids, hyperbolic spaces, and tori with different metrics (flat, embedded).

    The main results are \Cref{thm:absolute-continuity-drift} and  \Cref{thm:invariance-of-GPs}.  We illustrate our results in a statistical application where we estimate a parameter in the vector field $V$ from discrete-time observations of the process. 
    Compared to earlier work this has the following advantages: {\it (i)} the guiding term is  smooth  and thereby avoids analytical and numerical issues caused by discontinuities at the cut locus of the target;
         as a consequence of this, the numerics for computing the likelihood ratio between true conditioned process and guided process is easier as we avoid computing local times; {\it (ii)}
 no  Lyapunov condition (that is hard to verify) needs to be imposed.

\subsection{Outline}
     In \Cref{sec:SDEmanifolds}, we give some background on SDEs on manifolds. All main results of the paper are stated in Sections \ref{sec:mainresults} and \ref{sec:comparison}. We detail the approach taken in this paper in \Cref{sec:mainresults} in the specific case that $\Phi$ is the identity. In \Cref{sec:comparison} we extend this result to a case where $\Phi$ is a diffeomorphism between manifolds $M$ and $N$ and the heat kernel is known on $N$. In \Cref{sec:numerical}, we illustrate the usefulness of our results in numerical simulations of diffusion bridges and a parameter estimation problem for a discretely observed diffusion process on the torus and the Poincar\'e disk. In \Cref{sec:extensions}, we show how our approach can be extended to other types of conditionings at a fixed future time. Finally, all proofs are gathered in \Cref{sec:proofs}.

\section{Background on SDEs on manifolds}
\label{sec:SDEmanifolds}
We consider SDEs on manifolds as developments of $\bb{R}^d$-valued semimartingales. We sketch the main concepts here. For a detailed description, we refer to \cite{Hsu2002, emeryStochasticCalculusManifolds1989}. With $M$ a $d$-dimensional Riemannian manifold and $x\in M$, the tangent space $T_xM$ to $M$ at $x$ is a $d$-dimensional vector space. A \textit{frame} $u$ is a basis for $T_xM$, which can be seen as an isomorphism $u\colon \bb{R}^d \to T_xM$ where an $\bb{R}^d$-valued vector is mapped to $T_xM$ as a linear combination of the basis vectors in $u$. The collection of all frames at $x$ is denoted by $\mathrm{F}_xM$ and the vector bundle with fibers $\mathrm{F}_xM$ is called the \textit{frame bundle} $ \mathrm{F}M$. The canonical projection, which maps a frame $u$ at $x$ to $x$, is denoted by $\pi\colon \FM \to M$. In short, each element of $\FM$ can be interpreted as a tuple $u=(x,\nu)$ with $x\in M$ and $\nu = (\nu_1,\dots,\nu_d)$ where $\nu_1,\dots,\nu_d\in T_xM$ form a basis for the vector space $T_xM$. The canonical projection $\pi$ is simply the map $(x,\nu)\mapsto x$. The development $\bb{R}^d$-semimartingales onto $M$ refers to a a map that first maps the semimartingale to $\mathrm{F}M$ and then projects it onto $M$. The reason why this is convenient is that the increment $dZ_t$ of an $\bb{R}^d$-semimartingale is in $\bb{R}^d$ and can thus, through a frame, be mapped to the tangent space of $M$, which then gives a direction for the $M$-valued process to move to.

It can be shown that $\FM$ is a manifold itself of dimension $d+d^2$. It has a submanifold $\OM$ consisting of all the orthonormal frames, i.e. frames $(x,\nu)$ of which the vectors in $\nu$ are orthonormal with respect to the Riemannian inner product in $T_xM$. Given $u\in \mathrm{F}M$, we have to study the tangent space $T_u\FM$ in order to make sense of stochastic increments. Since the tangent space of a manifolds consists of derivatives of curves on the manifold, we consider $\FM$-valued curves first. A curve in $\FM$ can be decomposed in a horizontal part, which corresponds parallel transport of the basis along the projection onto $M$, and a vertical part, which is the part where the projection stays constant but only the basis changes over time. This decomposition extends to the derivative of a curve and therefore, for $u\in \FM$, $T_u\FM$ can be written as a direct sum $T_u\FM=\mathcal{H}_u\FM \oplus \mathcal{V}_u\FM$ of the horizontal and vertical subspaces, respectively. The projection $\FM\to M$ can be extended to a projection $T_u\FM \to T_{\pi u}M$, but what is interesting is that, when restricting ourselves to $\mathcal{H}_u\FM$, this projection is bijective, and can thus be inverted to a ``lift" $T_xM \to T_u\FM$, where $u$ is a frame with $\pi u=x$. It is therefore very useful to look at the horizontal subspace, as the vertical part of the process is irrelevant with us mainly being interested in the projection onto $M$. The horizontal subspace is a vector space of dimension $d$ with a basis $H_1(u),\dots, H_d(u)$ consisting of the globally defined ``horizontal fields" $H_1,\ldots,H_d$ that in local coordinates at $u = (x^i, \nu^i_j)$ are given by 
\[ H_i(u) = \nu^j_i \frac{\partial}{\partial x^j} - \nu^j_i \nu^l_m \Gamma^k_{jl}\frac{\partial}{\partial \nu^k_m}, \]
where  $\Gamma^k_{jl}$ denote Christoffel symbols on $M$. An $\bb{R}^d$-valued semimartingale $Z$ can be uniquely mapped to a horizontal process $U$ via the Stratonovich SDE 
\[ \dd U_t = H_i(U_t)\circ\dd Z_t^i.\]
The process $U$ is called the development of $Z$ and the projection $\pi(U)$ is the development of $Z$ on $M$.
The development is invertible in the sense that, for a horizontal semimartingale $U$, a unique $\bb{R}^d$-valued semimartingale $Z$ exists so that $U$ is the development of $Z$. This $Z$ is denoted the \textit{anti-development} of $U$. It can be shown that if $W$ is an $\bb{R}^d$-valued Brownian motion and $U_0=u_0\in\OM$, then $\pi(U)$ is Brownian motion in $M$. 

If $X$ is a semimartingale on $M$, a horizontal semimartingale $U$ such that $\pi(U)=X$ is referred to as a horizontal lift of $X$. For a function $f \colon M\to \bb{R}$, we denote by $\tilde{f}\colon\FM \to\bb{R}$ the function $f\circ\pi$. For $u\in \FM$, a vector $V\in T_{\pi(u)}M$ can be uniquely lifted to $\mathcal{H}_u\FM$ using the horizontal vector fields. This vector is then denoted by $\tilde{V}$. In particular, for a function $f\colon M\to\bb{R}$, we denote the lift of $\nabla f$ by $\nabla^H \tilde{f} = \left\{ H_i\tilde f , \dots, H_d\tilde{f} \right\}$ and refer to it as the horizontal gradient.

 \section{Guided processes on manifolds}
	\label{sec:mainresults}
In this section, we introduce guided processes taking values on a manifold $M$. We here assume that on $M$ a closed form expression for the heat kernel, possibly a series expansion, is known. In \Cref{sec:comparison}, we treat the more general situation of a manifold $N$ being diffeomorphic to $M$. 
    
	Throughout, we assume that we are working on a probability space $(\Omega,\mathcal{F}, \bb{P})$. Let $X$ be an $M$-valued Markov process with infinitesimal generator $\scr{L}$  and denote $\scr{A} = \partial_t+\scr{L}$. Let $\mathcal{F}_t = \sigma\left(\left\{X_s : s\leq t\right\}\right)$ be the natural filtration induced by $X$ and denote the restriction of $\bb{P}$ to $\mathcal{F}_t$ by $\bb{P}_t = \left.\bb{P}\right\rvert_{\mathcal{F}_t}$. We denote by $\mathcal{M}(M)$ the space of measurable functions from $[0,T)\times M$ to $\bb{R}$. 
      Define $\{E_t^f(X)\colon t \in [0,T)\}$ by
    \begin{equation}\label{eq:def-Mtf-Etf}		
        E_t^f(X) = \frac{f(t,X_t)}{f(0,X_0)}\exp\left(-\int_0^t \frac{\scr{A}f(s,X_s)}{f(s,X_s)}\dd s\right) 
    \end{equation}
    and let
        \[  \mathcal{H} = \left\{f\in\mathcal{M}(M)\: \colon f>0\text{ and }\{E_t^f(X)\,\colon t \in [0,T)\}\text{ is a martingale}\right\} .\]
		For $f\in\mathcal{H}$, the change of measure
	\begin{equation}
		\label{eq:change-of-measure}
		\der{\bb{P}_t^f}{\bb{P}_t}=E_t^f(X), \quad t \in [0,T)
	\end{equation}
 defines a family of probability measures $\left\{\bb{P}^f_t,\, t\in[0,T)\right\}$. Since \eqref{eq:change-of-measure} uniquely determines a family of consistent finite dimensional distributions, Kolmogorov's extension theorem implies that there exists a unique probability measure $\bb{P}^f$ such that, for all $t$, $\left.\bb{P}^f\right\rvert_{\mathcal{F}_t}=\bb{P}_t^f$. 
	We refer to the measure $\bb{P}^f$ as the measure induced by $f$. Expectations with respect to $\bb{P}^f$ are denoted by $\bb{E}^f$.

Specific choices of $f$ correspond to conditioning the process $X$ on certain events. Harmonic functions $h \in \mathcal{H}$ with the additional property $\mathfrak{A}h = 0$ play an important role. For such $h$, $\mathbb{P}^h$ is called Doob's $h$-transform of $\PP$.

We concentrate on conditioning the process on the event $
\{X_T=x_T\}$ with $x_T \in M$, postponing other types of conditioning to \Cref{sec:extensions}.  
We assume  the transition kernel of $X$ admits a  density $p$ with respect to the volume measure $\dd \mathrm{Vol}_M(x)$ on $M$, i.e.\ $\bb{P}\left(X_t \in A \mid X_s=x\right) = \int_A p(s,x;t,y)\dd \mathrm{Vol}_M(y)$ for $s\leq t$, $A\subseteq M$ measurable and $x\in M$. 
	Then,  $h(t, x) = p(s,t;T,x_T)$ satisfies  $\mathfrak A h= 0$ and induces the process $ \left(X\mid X_T=x_T\right)$ (see e.g Theorem 5.4.1 of \cite{Hsu2002}).

	Since transition densities for diffusions on manifolds are rarely known in closed form, we use the results of \cite{corstanje2021conditioning} to work around not knowing $h$ by replacing it with a tractable substitute, which we will denote by $g$.  This gives the following result.
  \begin{prop}\label{prop:likelihood}
  		Suppose $h\in\mathcal{H}$ is such that $\scr{A}h=0$ and let $g\in\mathcal{H}$. Then, for $t<T$,  $\bb{P}_t^h \ll \bb{P}_t^g$ with
		\begin{equation}
			\label{eq:absolute-continuity-halfopen}
			\der{\bb{P}^h_t}{\bb{P}^g_t} = \frac{E_t^h(X)}{E_t^g(X)} = \frac{g(0,X_0)}{h(0,X_0)}\frac{h(t,X_t)}{g(t,X_t)}\exp\left(\int_0^t \frac{\scr{A}g(s,X_s)}{g(s,X_s)}\dd s\right).
		\end{equation} 
  Under $\bb{P}^g$, the horizontal lift of $X$ satisfies the SDE
  \begin{equation} 
  \label{eq:SDE-conditioned-process}
    \dd U_t = \tilde{V}(U_t) \dd t + \nabla^H \log\tilde g(t, U_t) \dd t + H_i(U_t)\circ\dd W_t^i ,
    \end{equation}
    where we again used Einstein's summation convention. Alternatively, one could also write $U$ under $\bb{P}^h$ as the stochastic development of an $\bb{R}^d$-valued process through 
     \[ \dd U_t = H_i(U_t) \circ \left[ U_t^{-1}\tilde {V}(U_t)\dd t + U_t^{-1}\nabla^H \log\tilde{g}(t, U_t) \dd t + \dd W_t^i \right]\]
  \end{prop}
\begin{proof}
\Cref{eq:absolute-continuity-halfopen} follows directly from \eqref{eq:change-of-measure}. The SDE for $U$ follows from its extended generator, which can be obtained from \Cref{thm:PR}. 
\end{proof}

We call the process $X$ under $\mathbb{P}^g_t$ the {\it guided process}, as an appropriate choice of $g$ should guide the process to the event we wish to condition on. 
As an intermediate result, we obtain a Feynman-Kac type formula for the intractable function $h$ in terms of a process that is amenable to stochastic simulation. 
\begin{cor}
 Let $h(t,x) = \int p(t,x;T,y)q(y)\dd \mathrm{Vol}_M(y)$ such that
 $\PP^h \ll \PP$. Then $h$ is characterised in terms of the guided process by
       \begin{equation}
			\label{eq:h}
			{h(t,x)} = \mathbb E^g\, \left[ \left.{g(t,x)}\frac{q(X_T)}{g(T,X_T)}\exp\left(\int_t^T \frac{\scr{A}g(s,X_s)}{g(s,X_s)}\dd s\right)\right| X_t = x\right].
		\end{equation} 
\end{cor}

 Consider $U$ as defined by \eqref{eq:SDE-UV}. 
  The process $X = \pi(U)$ has extended generator $\scr{A}$ on the domain $\scr{D}(\scr{A}) = \mathcal{C}^{1,2}([0,T)\times M)$ with
		\[ \scr{A} f = \partial_t f + {V}{f} + \frac12 \Delta f \]
       for $f\in\scr{D}(\scr{A})$ and  where $\Delta$ denotes the Laplace-Beltrami operator on $M$.

\subsection{Guiding induced by the heat kernel} 
 Let $\kappa$ be the minimal heat kernel on $M$. That is $\kappa\colon[0,T]\times M\times M \to \bb{R}$ is the unique minimal solution to the partial differential equation
		\begin{equation} 
			\label{eq:heat-equation}
			\partial_t \kappa(t,x,y) = \frac12\Delta_x\kappa(t,x,y)
		\end{equation}
	satisfying $\int_M \kappa(t,x,y)\dd \mathrm{Vol}_M(y) = 1$.  Here $\Delta_x$ denotes the Laplace-Beltrami operator applied to the $x$-argument. Notice that, $\kappa$ exists on compact manifolds, as they are stochastically complete. We define the guiding function
		\begin{equation}\label{eq:heatkernel} g(t,x)= \kappa(T-t,x, x_T) \end{equation}
  By Proposition 4.1.6 in  \cite{Hsu2002}, $g$ is the transition density function of a Brownian motion on $M$. We require $X$ to have a transition density $p$ as well. The relationship between $p$ and the heat kernel is described in \Cref{ass:transitiondensity}. 
		\begin{prop}
		\label{ass:transitiondensity}
			Let $M$ be a compact Riemannian manifold. Then $X$ admits a transition density $p$ with respect to the volume measure for which positive constants $\gamma_1,\gamma_2,\gamma_3$ exists such that for all $s\leq t$ and $x,y\in M$, 
			\[ p(s,x;t,y) \leq \gamma_1 \kappa(\gamma_2(t-s), x,y)e^{-\gamma_3 (t-s)} \]
		\end{prop}
   \begin{proof} 
   This result follows from Equation (1.2) of \cite{sturm1993} if $M$ has Ricci curvature bounded from below, which is the case for compact manifolds. 
   \end{proof}
   The following theorem gives sufficient conditions for equivalence of the measures $\bb{P}^h_T$ and $\bb{P}^g_T$. It is key to our numerical simulation results presented in \Cref{sec:numerical}.

		\begin{thm}
    \label{thm:absolute-continuity-drift}
		Assume $M$ is a compact manifold and $V$ is a smooth vector field on $M$. Let $h\in\mathcal{H}$ be such that $\scr{A}h=0$ and let $g$ be as in \eqref{eq:heatkernel}. 
        Then $\bb{P}_T^h \sim \bb{P}_T^g$ with 
		\begin{equation}
            \label{eq:radon-nikodym-manifolds}
		 \der{\bb{P}_T^h}{\bb{P}^g_T} = \frac{g(0,x_0)}{h(0,x_0)} \Psi_{T,g}(X),	
		\end{equation}
		where 
		\begin{equation}\label{eq:Psi}\Psi_{T,g}(X) = \exp\left(\int_0^T V\left(\log g\right)(s, X_s) \dd s \right).  \end{equation}
		\end{thm}
            The proof is deferred to \Cref{app:proof-of-mainthm}. 
            Upon integration of \eqref{eq:radon-nikodym-manifolds} with respect to $\bb{P}^g$, we obtain an expression for the transition density as
            \begin{equation}\label{eq:likelihood} p(0,x_0 ; T,x_T) = g(0,x_0)\bb{E}^g \Psi_{T,g}(X).
            \end{equation}
		
 \section{Guiding with comparison manifolds}
 \label{sec:comparison}
 
  The assumption that the heat kernel is known in closed form is restrictive. Suppose $N$ is a manifold for which there exists a diffeomorphism $\Phi\colon M\to N$. Assuming that the heat kernel on $M$ is tractable, we will construct a guided process on $N$ by mapping the guiding term from $M$ to $N$.

Let $X=\pi(U)$ be an $M$-valued stochastic process, with $U$ defined by \eqref{eq:SDE-UV}. \Cref{thm:absolute-continuity-drift} shows how samples from $X$ under $\bb{P}^h$ can be weighted to obtain samples from $\bb{P}^g$. Denote the process $X$ under $\bb{P}^h$ by $X^h$. 	 
  Define the $N$-valued process $Y$ by $Y_t = \Phi(X_t)$ and denote the law of $Y$ by $\bb{Q}$. Define  $\Phi_* h \in \mathcal{M}(N)$  by $(\Phi_* h)(t,y)=h(t, \Phi^{-1}(y))$.

The theorem below shows that first $h$-transforming $X$ and then mapping it to $N$ through $\Phi$ results in the same process (in law) as $h$-transforming the process $\Phi(X)$ with $\Phi_* h$. In other words, the diagram in \Cref{fig:diagram} commutes. 
\begin{thm}\label{thm:invar} 
Let $h \in \mathcal{H}$ and take $t <T$.
	 For $\psi \colon N\to \mathbb{R}$ bounded and $\mathcal{F}_t$-measurable
\[ \bb{E}\psi(\Phi(X^h)) = \bb{E}\left[\psi(Y)E_t^{\Phi_* h}(Y)\right]\]
\end{thm} 
The proof is deferred to \Cref{subsec:proofinvariance1}. 

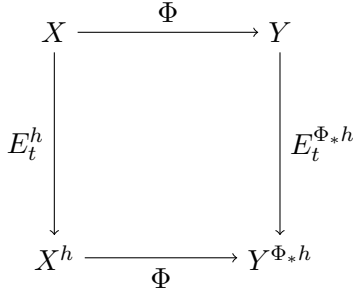
\begin{figure}
	\centering
		\begin{tikzpicture}
			\node at (0,0) (a) {$X$};
			\node at (3,0) (b) {$Y$};
			\node at (0,-3) (c) {$X^h$};
			\node at (3,-3) (d) {$Y^{\Phi_* h} $};
			\path (a) -- node[midway, left] (text) {$E_t^{h}$} (c);
			\draw [->] (a)--(c);
			\path (a) -- node[midway, above] (text3) {$\Phi$} (b);
			\draw [->] (a)--(b);
			\path (b) -- node[midway, right] (text2) {$E_t^{\Phi_* h}$} (d);
			\draw [->] (b)--(d);
			\path (c) -- node[midway, below] (text4) {$\Phi$} (d);
			\draw [->] (c)--(d);
		\end{tikzpicture}
		\caption{Diagram demonstrating that one can either first map $X$ to $\Phi(X)$ and then condition, or one can first condition $X$ and then map the conditioned process through $\Phi$. Horizontal and vertical arrows correspond to transforming by $\Phi$ and conditioning respectively. }
		\label{fig:diagram}
\end{figure}
Our next step is to establish an analogue of \Cref{thm:absolute-continuity-drift} for the guided process on $N$. To this end, we note that the proof of \Cref{thm:absolute-continuity-drift} is based on a result from \cite{corstanje2021conditioning}, which we have included for the reader's convenience in \Cref{sec:proofs} (see \Cref{thm:absolute-continuity}). It requires checking multiple conditions for functions $g$ and $h$ as above. The theorem below states that these conditions are then necessarily also satisfied for $\Phi_* g$ and $\Phi_* h$. This means that equivalence of the laws of the true conditioned process and the approximation using a guided process is maintained under diffeomorphisms. We believe this is a natural property which nevertheless is not shared by some other methods for diffusion bridge simulation (such as for example  \cite{delyon2006}).

 Let $\bb{Q}_t$ denote the measure $\bb Q$ restricted to $\sigma\left(Y_s\colon s\leq t\right)$. Define the law $\bb{Q}_t^{\Phi_* h}$ by the change of measure 
 \[\frac{\dd \bb{Q}_t^{\Phi_* h}}{ \dd \bb{Q}_t}(Y) = E_t^{\Phi_* h}(Y).\] 
Suppose $Y$ starts at $y_0$ and is conditioned on $y_T$ at time $T$. 
\begin{thm}
	\label{thm:invariance-of-GPs}
If the assumptions of \Cref{thm:absolute-continuity} are satisfied for the process $X$ with $g$ and $h$, then these assumptions are also satisfied for the process $Y$ with $\Phi_* g$ and $\Phi_* h$. Therefore, 
 	\[ \der{\bb{Q}_T^{\Phi_* h}}{\bb{Q}_T^{\Phi_* g}}(Y) = \frac{\Phi_* g(0,y_0)}{\Phi_* h(0,y_0)} \Psi_{T,g}(\Phi^{-1}(Y)),\]
  with $\Psi_{T,g}$ as defined in \eqref{eq:Psi}. 
\end{thm}
The proof is given in \Cref{subsec:proofinvariance2}. If $\tilde{\Phi}\colon \FM \to \mathrm{F}N$ denotes the horizontal lift of $\Phi$, then under $\bb{Q}^{\Phi_*g}$, the process $\Lambda_t := \tilde\Phi(U_t)$ satisfies the Stratonovich SDE
\begin{equation}
\label{eq:mapped-SDE}
    \dd\tilde\Phi(U_t) = \left(\tilde{\Phi}_*\tilde{V}\right)\left(U_t\right)\dd t + \tilde\Phi_*\left(\nabla^H \log \tilde g\right)\left(t, U_t\right) \dd t + \left(\tilde{\Phi}_*H_i\right)\left(U_t\right)\circ\dd W_t^i .
\end{equation}
The guiding term $\tilde\Phi_*\left(\nabla^H \log \tilde g\right)$ can be rewritten as $\left\{ \left(\tilde{\Phi}_*H_i\right)(\log\tilde \Phi_*\tilde g),\,i=1,\dots,d\right\}$. 
To sample processes under $\bb{Q}_T^{\Phi_* h}$, one can apply an Euler-Heun scheme to \eqref{eq:mapped-SDE}. However, computation of the pushforward vector fields can be avoided as in \Cref{sec:numerical} we give a Metropolis-Hastings algorithm to sample from $X\sim \bb{P}_T^h$ on $M$. \Cref{thm:invariance-of-GPs} then implies that such samples can be converted to samples $\Phi(X) \sim \bb{Q}_T^{\Phi_* h}$.
	\begin{ex}[Euclidean SDEs]
	\label{rem:InvarianceGuidingTerm}
		If $X$ is a Euclidean stochastic differential equation (SDE) with drift $b$, diffusion coefficient $\sigma$ and let $a = \sigma\sigma^\T$, then $X^h$ satisfies the SDE
	\[ \dd X^h_t = b(t, X^h_t) \dd t + a(t,X^h_t) r_h(t,X^h_t) \dd t + \sigma(t, X^h_t) \dd W_t, \]
	 where $r_h(t,x) = \nabla_x\log h(t,x)$. 
It follows from direct computations that the guiding term for $Y^{\Phi_* h}$ is given by $a^\Phi(t,y)r_{\Phi_* h}(t,y)$, where
		\begin{align*}
			 r_{\Phi_* h}(t,y) &= \nabla_y\log \Phi_* h(t,y) = J_\Phi({\Phi^{-1}}(y))^\T \left. \nabla_x\log h(t, x)\right\rvert_{x=\Phi^{-1}(y)} \\ &= J_\Phi({\Phi^{-1}}(y))^\T r_h(t,\Phi^{-1}(y))
		\end{align*} 
		and $a^\Phi(t,y) =\sigma^\Phi(t,y)\sigma^\Phi(t,y)^\T$ with  \[\sigma^\Phi(t,y) = J_\Phi(\Phi^{-1}(y))^\T \sigma(t,\Phi^{-1}(y)).\] In view of \eqref{eq:mapped-SDE}, this is what one should expect, as $r(t,x)$ is a vector in $T_x \bb{R}^d$ and the new guiding term $\Phi_*\left(\nabla_x \log h\right)(t, \Phi^{-1}(y))$ is exactly the pushforward of this tangent vector under $\Phi$. 
	\end{ex}

    	\section{Numerical simulations}
            \label{sec:numerical}
            For numerical simulations of SDEs on manifolds, we utilize the functionalities implemented in \cite{manifolds}, to which we added functions for horizontal development similar to the Python implementation described in \cite{kuhnel2017computational}. Julia written code is available at \cite{manifold_bridge}. 

        \subsection{Bridge simulation}
Suppose we wish to sample bridges on the manifold $N$ connecting $y_0$  at time $0$ and $y_T$ at time $T$. Suppose $\Phi \colon M \to N$ is a diffeomorphism  and on $M$
the heat kernel is known in closed form.  Assume  $g$ as in \eqref{eq:heatkernel} with $x_T=\Phi^{-1}(y_T)$ and let $x_0=\Phi^{-1}(y_0)$. Let $u_0$ be such that $x_0=\pi(u_0)$. 
Assuming a strong solution to \Cref{eq:SDE-conditioned-process}, there exists  a map $\mathcal{GP}$ such that  $U=\mathcal{GP}(u_0, Z)$. In our numerical implementation, this means that the SDE for $U$ is solved by Euler-Heun discretisation on a dense (fixed) grid, and $Z$ is the vector with increments of the driving Wiener process of the SDE on this grid. With slight abuse of notation we write $X=\pi(U)$ for the projection of $U$ on the manifold $M$. 

\Cref{alg:crank-nicolson} gives a Metropolis-Hastings algorithm to sample from $\bb{P}^h$ on $N$ using a preconditioned Crank-Nicolson step (\cite{Cotter2013MCMC}).  With $\lambda=0$ it corresponds to an Metropolis-Hastings algorithm with independent proposal distribution.
Note that the guided process does take the drift $V$ into account, but ignores it in the guiding term. If the drift is strong, or we condition the process on a very unlikely point, the acceptance rate of the independence sampler may be very low. In that case, taking a value of $\lambda$ closer to $1$, for example to target a Metropolis-Hastings acceptance percentage of about $25\%-50\%$, will facilitate the chain to reach stationarity.  In the numerical simulation results that we present, $\lambda$ was chosen by monitoring the acceptance rates on a trial run and subsequently adjusted to achieve the desired acceptance rate. Alternatively, adaptive methods can be used for tuning $\lambda$ (cf.\ \cite{roberts2009examples}). In the simulation results on the torus, we sampled the guided process on a grid with mesh width $0.001$, mapped through a time change $t\mapsto t(1.7 - 0.7t/T)$ to allow for more grid points near $T$.  

\begin{algorithm}[h]
    		    \caption{Metropolis-Hastings algorithm for sampling diffusion bridges on manifold $N$}
    		    \label{alg:crank-nicolson}
    		    \KwIn{Number of iterations $K$ and  tuning parameter $\lambda\in[0,1)$. }
                Sample a Brownian motion $Z$, set $U=\mathcal{GP}(u_0, Z)$ and $X=\pi(U)$ \;
                Sample a Brownian motion $W$ independent of $Z$ and let $Z^\circ = \lambda Z + \sqrt{1-\lambda^2}W$ \; 
                Set $U^\circ =  \mathcal{GP}(u_0, Z^\circ)$, set $X^\circ=\pi(U^\circ)$ and compute 
                $ \alpha =1 \wedge \Psi_{T,g}(X^\circ)/ \Psi_{T,g}(X)$ with $\Psi_{T,g}$ defined in \eqref{eq:Psi} \;
                With probability $\alpha$ set $X\gets X^\circ$ and $Z\gets Z^\circ$. \;
                Repeat steps 2--4 $K$ times to obtain $X^{(1)}, \ldots, X^{(K)}$\;
                Output $\Phi(X^{(1)}),\ldots, \Phi(X^{K)})$.
\end{algorithm}

     \subsection{Parameter estimation}
        \label{subsec:parameter_estimation}
Suppose we have observations $\mathcal{D} = \{ X_{t_i} = x_i,\, i=1,\dots, n\}$ and wish to infer the parameter $\theta$ in \eqref{eq:vector_field_2_terms}. To perform parameter inference conditional on $\mathcal{D}$, we use data augmentation, where we construct a Markov chain that samples from the joint distribution of $(\theta , X)$, conditional on $\mathcal{D}$. Here, $X=\{X_t,\, t\in [0,t_n]\}$.  Specifically, we repeatedly perform the following steps 
            \begin{enumerate}
                \item  Update $Z \mid (\mathcal{D}, \theta)$ using steps 2--4 of  \Cref{alg:crank-nicolson} and set $U = \mathcal{GP}_\theta(u_0,Z)$. 
                \item Update $\theta \mid (\mathcal{D}, X)$, where $X=\pi(U)$. 
                \item Compute $Z$ as the anti-development of $X$ under the  drift with parameter $\theta$. That is,
                \[ Z_t = \int_0^t U_s^{-1}\circ\left(\dd X_s - V_\theta(X_s)\dd s - \nabla\log g(s,X_s)\dd s\right) \] 
                \end{enumerate}

Step (2) requires an update of $\theta$. To do this, we utilize the following conjugacy result, based on Proposition 4.5 of \cite{mider2021continuous} for normal priors. It follows from Girsanov's theorem, see e.g. Proposition 1B of \cite{elworthy1988geometric} , that if we assign a prior $\pi(\theta)$ to $\theta$ then
             \[ \pi(\theta \mid Z, X) \propto \exp\left(\int_0^{t_n} \inner{V_\theta(X_s)}{\dd X_s}_{X_s} - \frac12 \int_0^{t_n} \inner{V_\theta(X_s)}{V_\theta(X_s)}_{X_s} \dd s \right)\pi(\theta). \]
            Hence, if  $\theta \sim N(0, \Gamma_0^{-1})$, then $\theta\mid X, Z \sim N(\Gamma^{-1}\mu, \Gamma^{-1})$, where,
            \[\begin{aligned}  \mu &= \Bm \int_0^{t_n} \inner{\phi_{k}(X_s)}{\dd X_s}_{X_s} \Em_{k=1}^K \\
            \Gamma &= \Bm \int_0^{t_n} \inner{\phi_{k}(X_s)}{\phi_{\ell}(X_s)}_{X_s} \dd s \Em_{k,\ell=1}^K + \Gamma_0
            \end{aligned} \]

\subsection{Example: guided processes on the torus}
            We illustrate the simulation of diffusion bridges on the embedded torus. To do this, we use the known closed form expression for the heat kernel on the flat torus and then map that to the embedded torus in $\bb{R}^3$ with the metric inherited from the embedding space. From Example 6 of \cite{hansen2021geometric}, it can be deduced that the heat kernel on the two-dimensional flat torus $\bb{T}^2 = \bb{S}^1 \times \bb{S}^1$ is given by 
            \begin{equation}
            \label{eq:heat-kernel-torus}
                \kappa(t,x,y) = \frac{1}{2\pi t} \sum_{k\in\bb{Z}}\sum_{\ell\in\bb{Z}} \exp\left(-\frac{(x_1-y_1-2k\pi)^2}{2t} - \frac{(x_2-y_2-2\ell\pi)^2}{2t}\right).
            \end{equation}
                For $\theta \in \bb{R}^2$, we impose the vector field     
            \begin{equation}
                \label{eq:vector_field_2_terms}
                V_\theta(x) = \theta_1\phi_1(x) + \theta_2\phi_2(x),\qquad x\in M,
            \end{equation}
            where $\phi_1$ and $\phi_2$ are the vector fields sketched in \Cref{fig:vector_fields}. In the numerical implementation we truncated the series expansion  \eqref{eq:heat-kernel-torus} to the indices $-10, \dots, 10$ of both series.

            \begin{figure}[h]
            \centering
                \includegraphics[width = 2.5in]{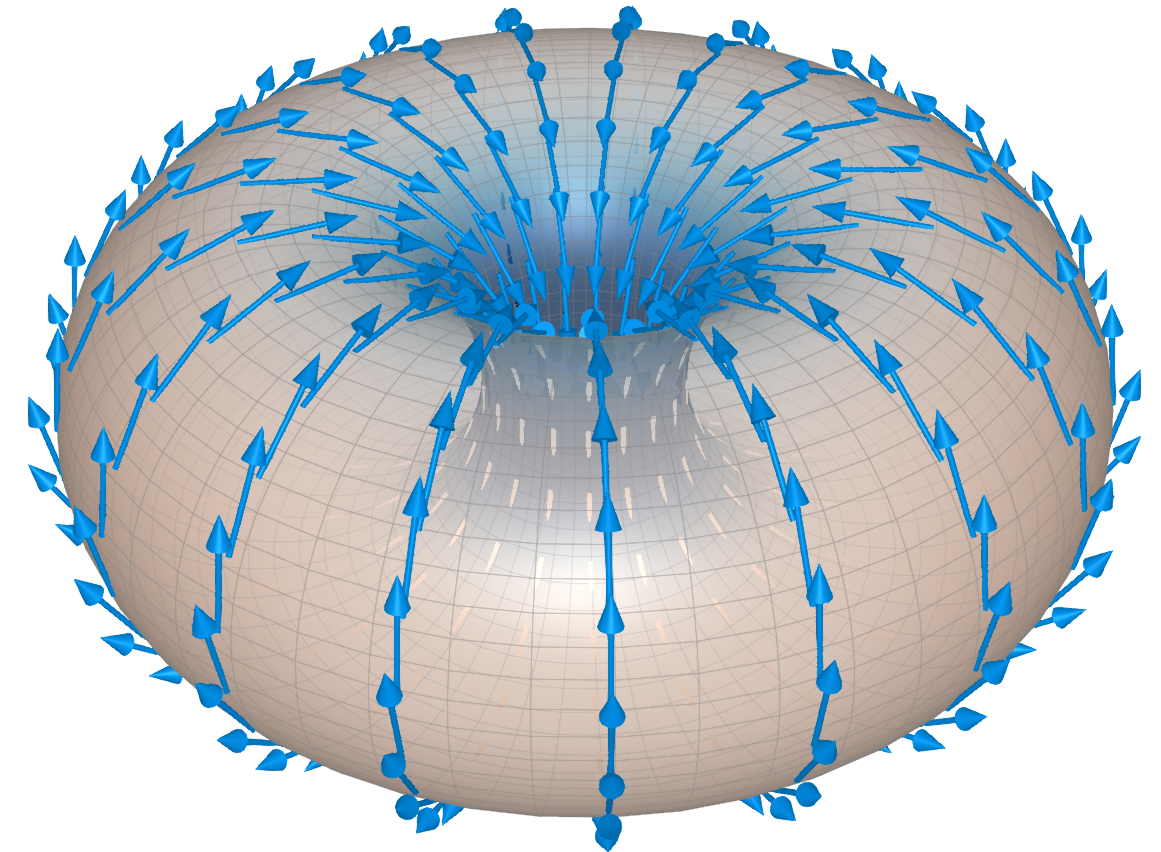}
                \includegraphics[width = 2.5in]{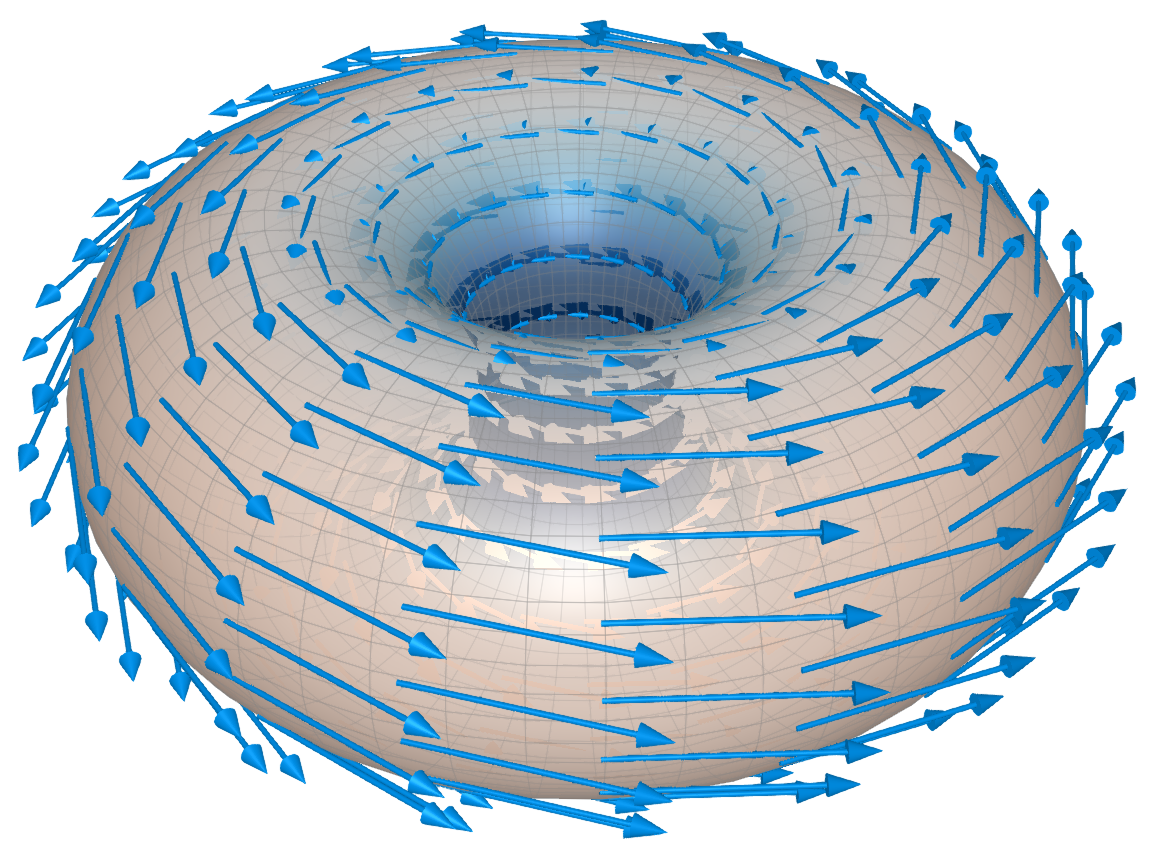}
                \caption{ Illustration of the vector fields $\phi_1$ (left) and $\phi_2$ (right).}
                \label{fig:vector_fields}
            \end{figure}
We apply \Cref{alg:crank-nicolson} in the following settings:
\begin{enumerate}
    \item A somewhat weak vector field with $\theta=(-2,2)$, where we condition the process on a point obtained from forward simulation of the unconditional process (\Cref{fig:bridge-with-forward-path_-22}, guided processes obtained with $\lambda=0.95$). One can visually assess that samples of the guided process resemble the original forward path.
       \item A stronger vector field with $\theta=(-1,8)$, where we condition the process on a point obtained from forward simulation of the unconditional process (\Cref{fig:bridge-with-forward-path_-18}, guided processes  obtained with $\lambda = 0.99$). Due to the strong drift, paths should take a turn around the torus. Indeed, samples of the  guided process do so. 
    \item A very weak vector field with $\theta=(-1,0.1)$, where we condition the process on a point which is unlikely under the forward  unconditional process (\Cref{fig:bridge-with-forward-path_-11}, guided processes  obtained with $\lambda = 0.7$). In this case, samples from the guided process can either go clockwise or counterclockwise, the latter seemingly more likely. In this case, in $\Cref{tab:acceptance rates}$ we also report how the choice of $\lambda$ affects the acceptance rate in \Cref{alg:crank-nicolson}.  
\end{enumerate}
In all cases the right-hand-plot displays every 50-th sample obtained from 1000 iterations of \Cref{alg:crank-nicolson}. 

            \begin{figure}[h]
            \centering
                \includegraphics[height = 1.7in]{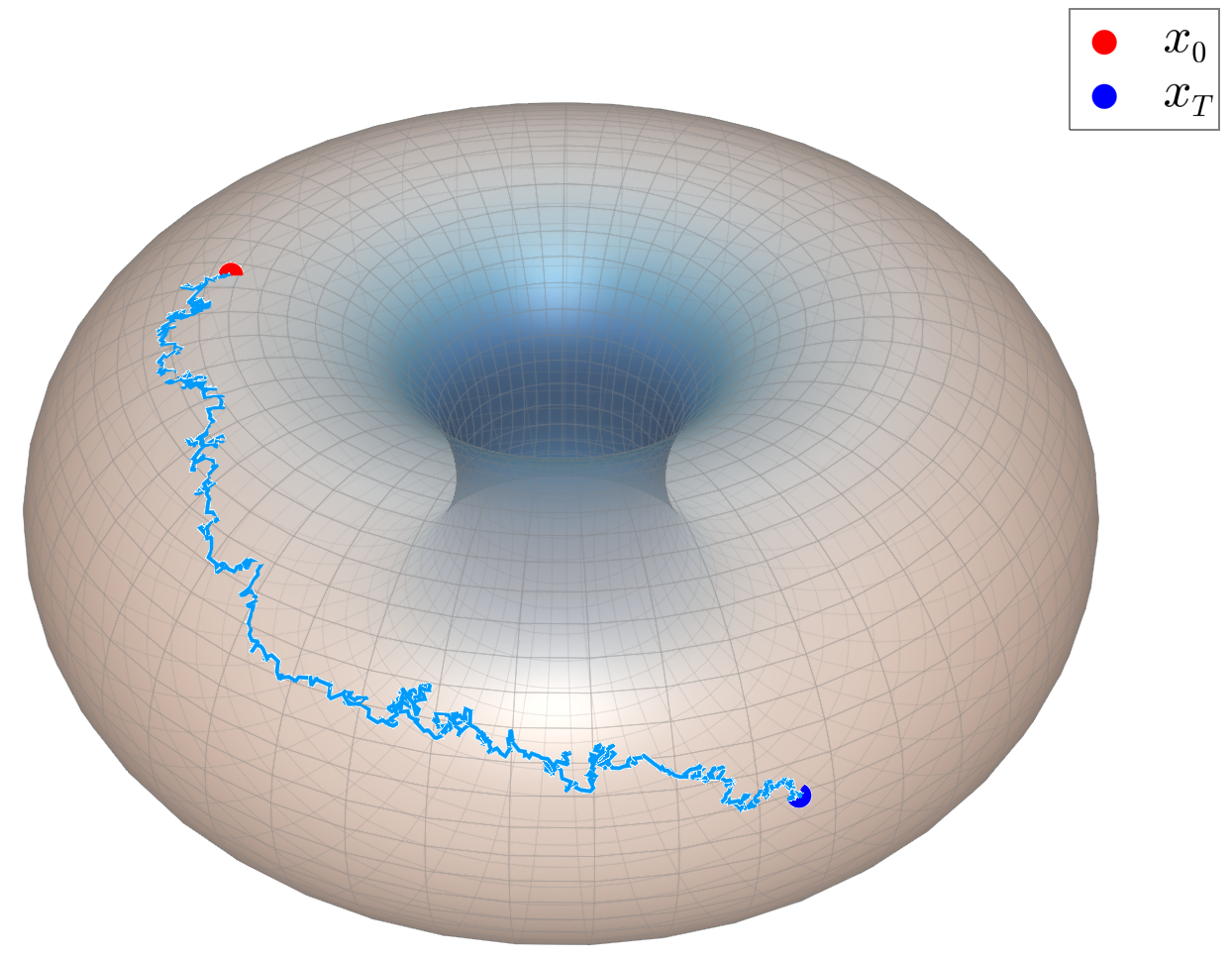}
                \includegraphics[height = 1.7in]{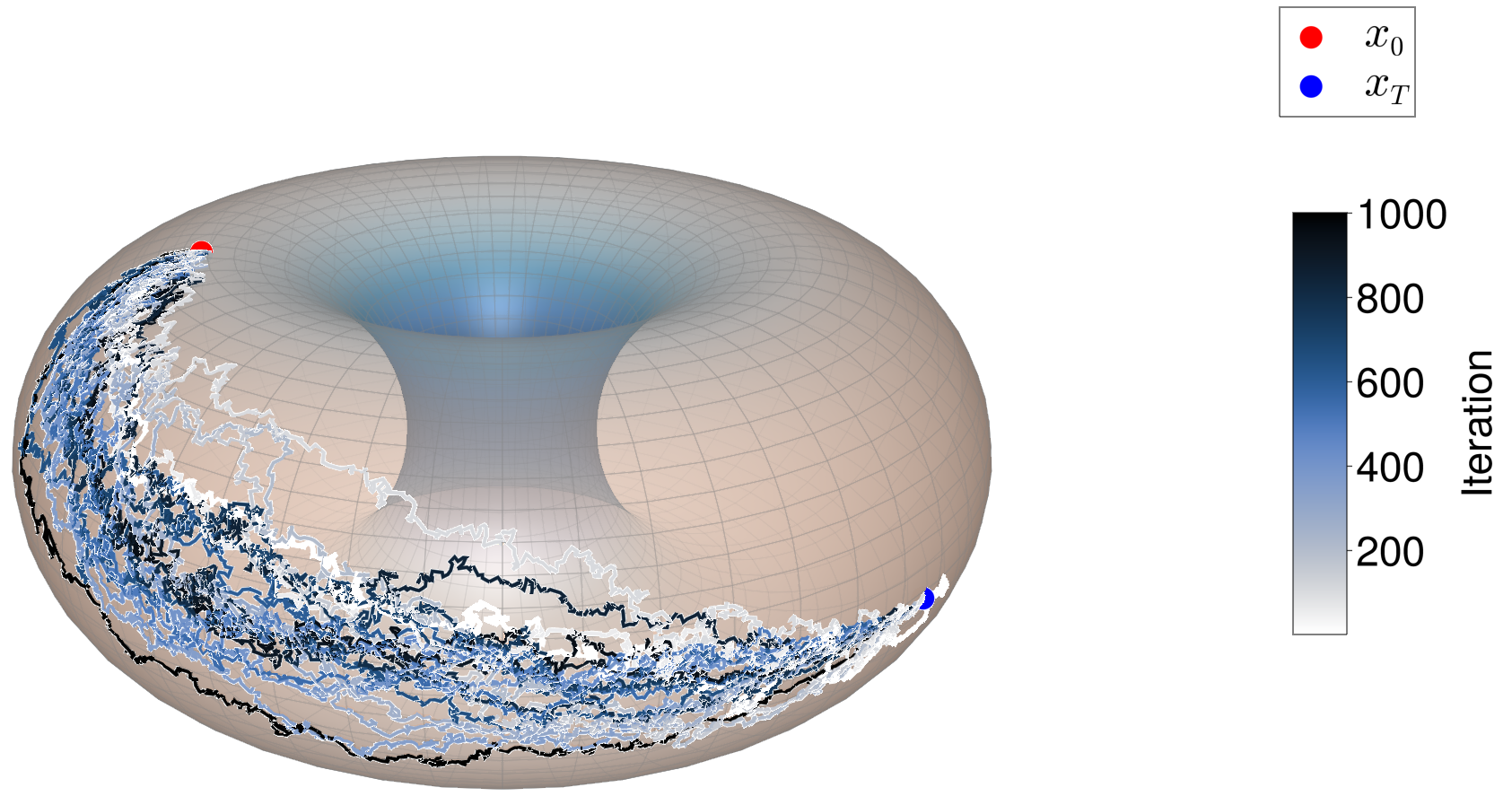}
                \caption{Left: Forward simulation of unconditional process $X$. Right: final 6 accepted samplepaths of 1000 repetitions of \Cref{alg:crank-nicolson} of bridge processes starting at $(3,0,2)$ conditioned to hit $X_T$ at time $T=1$. We used the vector field described in \eqref{eq:vector_field_2_terms} with $\theta = (-2,2)$. The acceptance rate was 17\%.}
                \label{fig:bridge-with-forward-path_-22}
            \end{figure}

            \begin{figure}[h]
            \centering
                \includegraphics[height = 1.7in]{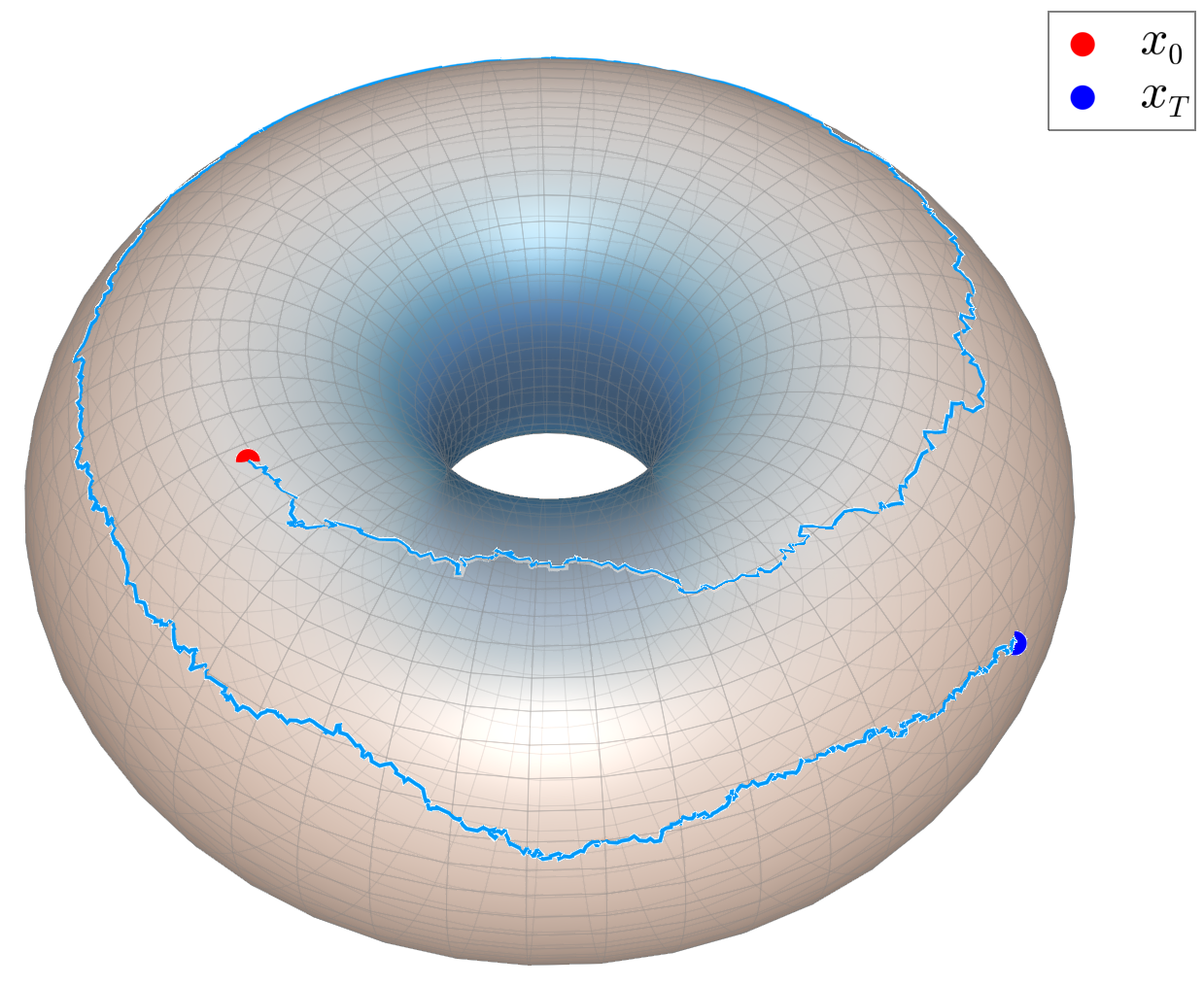}
                \includegraphics[height = 1.7in]{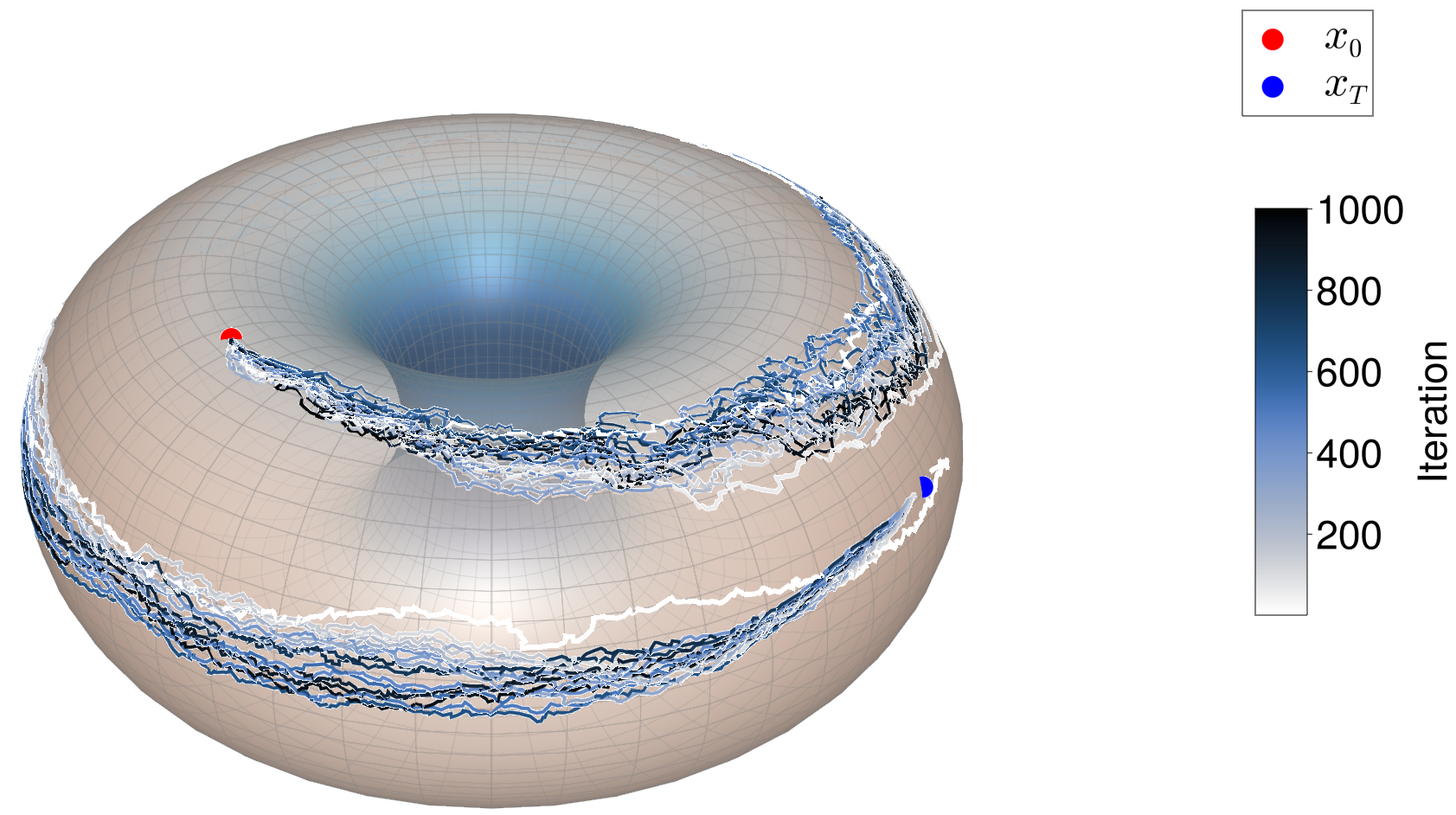}
                \caption{Left: Forward simulation of $X$. Right: final 6 accepted samplepaths of 1000 repetitions of \Cref{alg:crank-nicolson} of bridge processes starting at $(3,0,2)$ conditioned to hit $X_T$ at time $T=1$. We used the vector field described in \eqref{eq:vector_field_2_terms} with $\theta = (-1,8)$. The acceptance rate was 19.5\%.}
                \label{fig:bridge-with-forward-path_-18}
            \end{figure}            

            \begin{figure}[h]
            \centering
                \includegraphics[height = 1.7in]{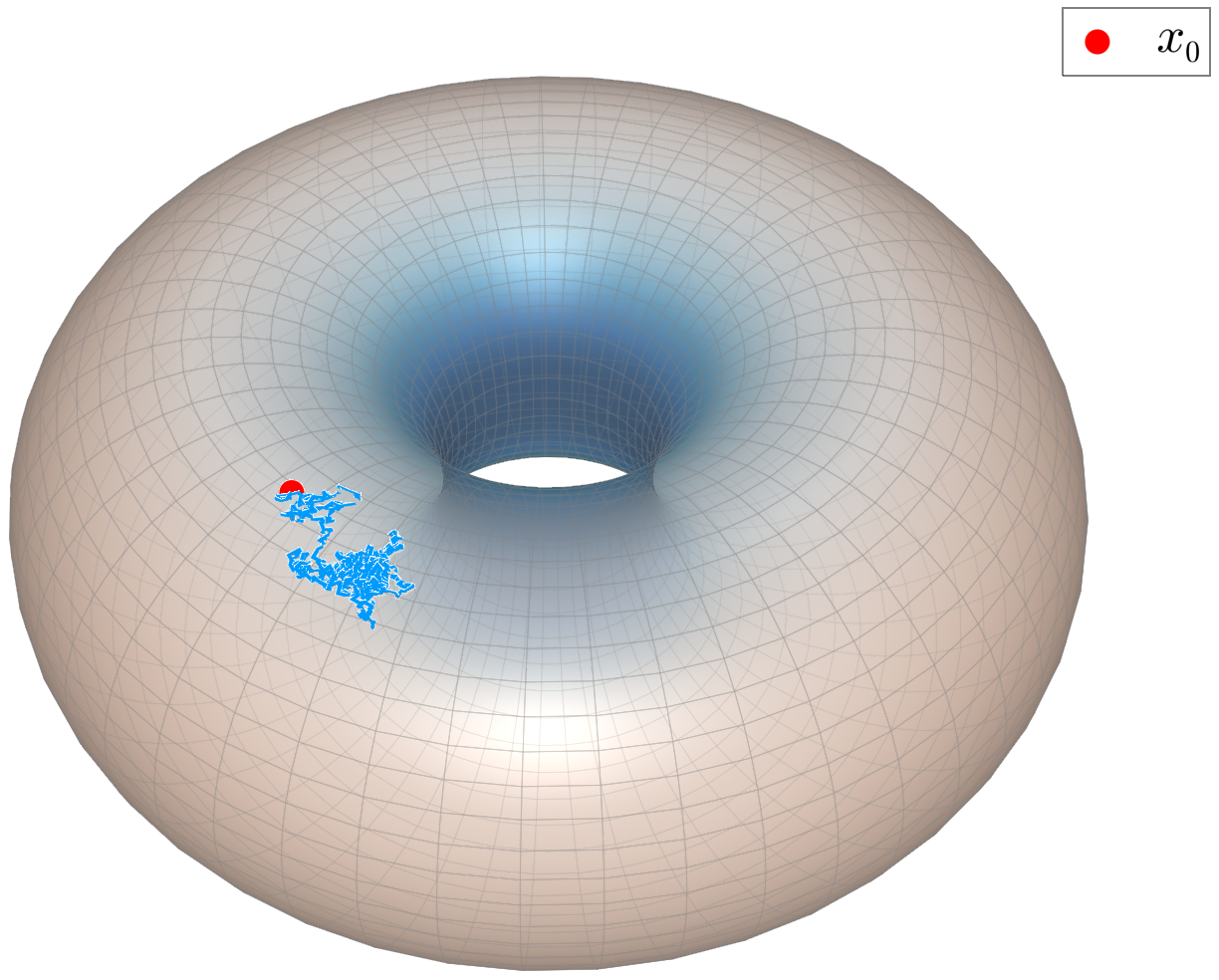}
                \includegraphics[height = 1.7in]{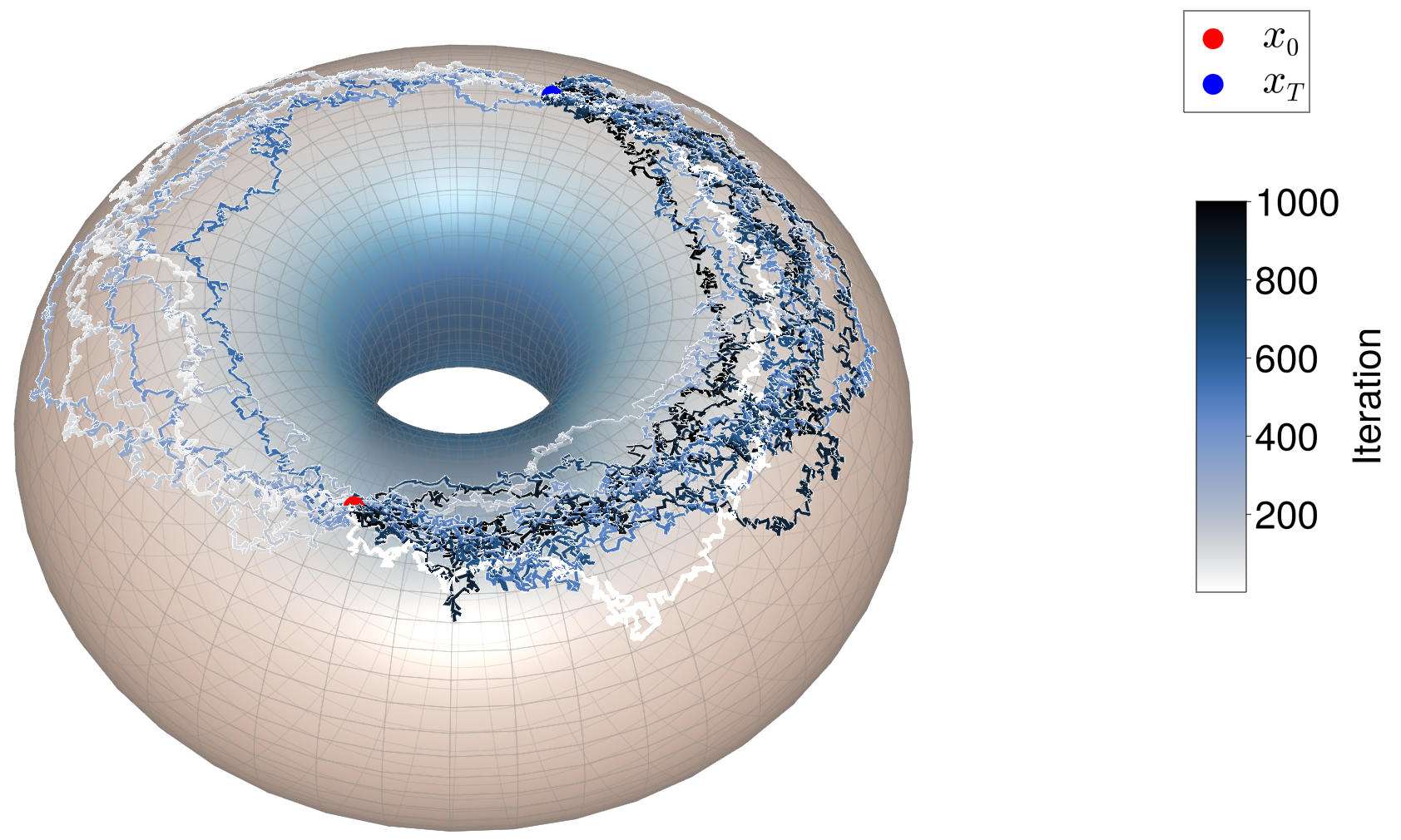}
                \caption{Left: Forward trajectory of the unconditioned process  Right: One trajectory for each 20 accepted proposals of \Cref{alg:crank-nicolson} for the process conditioned on hitting $(-3,0,2)$. A darker color indicates a later iteration of the algorithm. The vector field \eqref{eq:vector_field_2_terms} with $\theta = (-1, 0.1)$ was applied.}
                \label{fig:bridge-with-forward-path_-11}
            \end{figure}

             \begin{table}[h]
                 \centering
                 \begin{tabular}{l |c |c |c |c |c } \hline 
                     $\lambda$ & 0.0 & 0.5 & 0.7 & 0.9 & 0.99 \\ 
                     acceptance rate & 10\%  & 27\%  & 39\%   & 60\% &        89\%  \\ \hline 
                 \end{tabular}
                 \caption{Acceptance rate for various choices of $\lambda$ for the 3rd setting considered (\Cref{fig:bridge-with-forward-path_-11}).}
                 \label{tab:acceptance rates}
             \end{table}

We applied the algorithm for parameter estimation from \Cref{subsec:parameter_estimation}  to a dataset generated under forward simulating $X$ with vector field \eqref{eq:vector_field_2_terms} with $\theta=(4,-4)$ and saving its values at $40$ equidistant times. The posterior samples are displayed in \Cref{fig:gibbs}. For the prior, we used $\Gamma_0 = 0.05I$ and we used tuning parameter $\lambda = 0.9$ in the Crank-Nicolson step. 

            \begin{figure}[h]
            \centering
                \includegraphics[width = 4.5in]{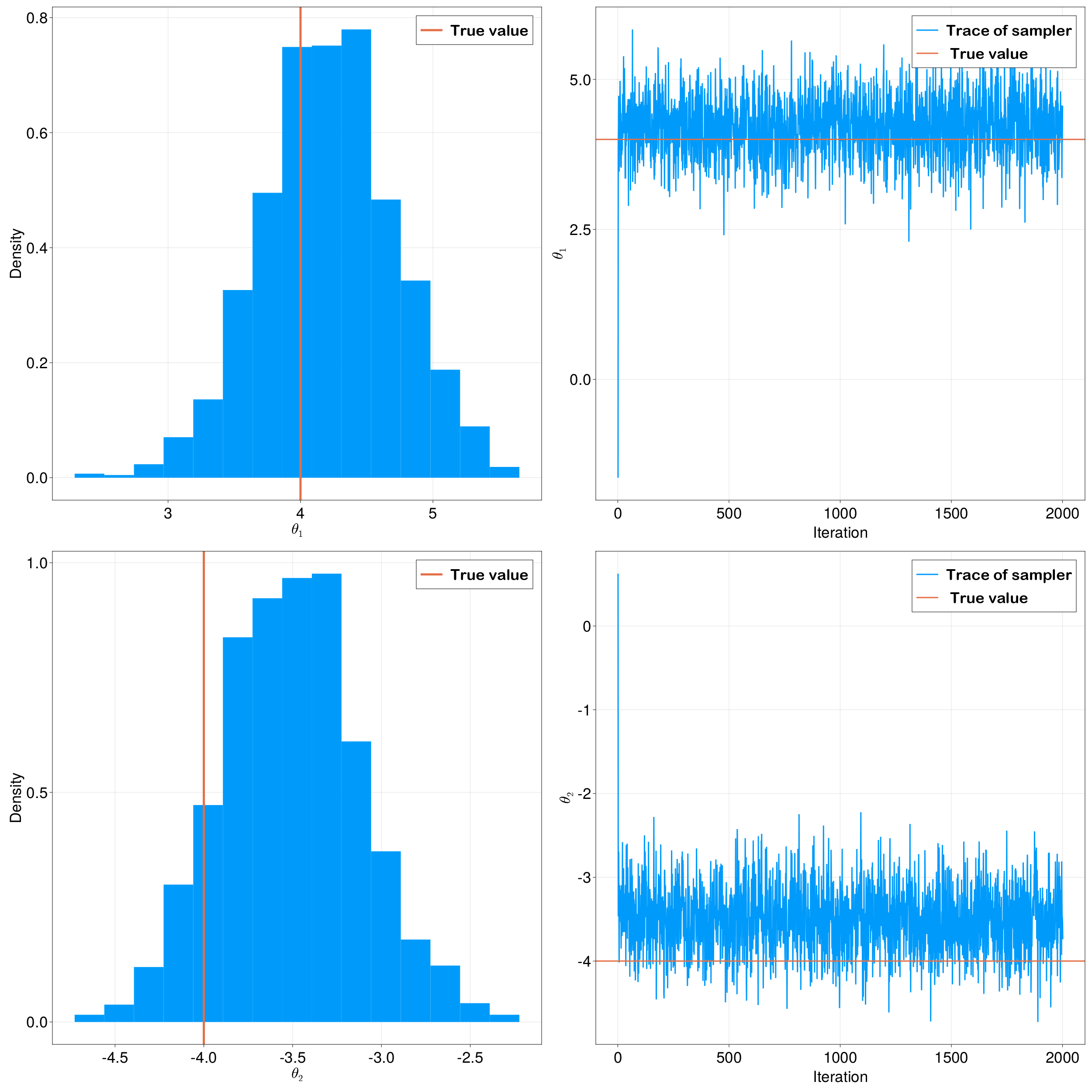}
                \caption{Left: Histograms 1900 samples from the posterior distribution of $\theta$ after 100 iterations of burn-in time using the Gibbs sampler Right: Trace of the sampler. The top row are the samples for $\theta_1$, the bottom row the samples for $\theta_2$. Data were generated with $(\theta_1,\theta_2)=(4,-4)$ as indicated by the red lines.}
                \label{fig:gibbs}
            \end{figure}

\subsection{Example: guided processes on the Poincar\'e disk}
\cite{hansen2021geometric} provide several examples of manifolds on which the heat kernel is known. On all manifolds diffeomorphic to these, the approach presented in this article can be applied. Since diffusion models on hyperbolic spaces are often seen in machine learning applications, see e.g. \cite{fu2024hyperbolic}, \cite{wen2024hyperbolicgraphdiffusionmodel} as well as in the theory of general relativity, we'll explore this setting in more detail here.

The hyperbolic space is a non-Euclidean geometric space with constant negative sectional curvature. Unlike Euclidean spaces, where parallel lines remain equidistant, in hyperbolic space, they diverge. The $2$-dimensional hyperpolic space can be embedded into t $\bb{R}^3$, equipped with the Minkowski metric by considering the subspace of all points with Minkowski norm $-1$. Here, we consider the Poincar\'e disk  to model the $2$-dimensional hyperbolic space in local coordinates, which can be interpreted as a stereographic projection of the three-dimensional representation onto the unit disk. Here the entire infinite space is mapped inside the unit disk, preserving angles but distorting distances. Let $D=\{z\in \mathbb{C}\, :\, |z|<1\}$. The hyperbolic distance between $z_1, z_2 \in D$ is given by 
\[
\dist(z_1, z_2) = \operatorname{arcosh} \left( 1 + \frac{2 |z_1 - z_2|^2}{(1 - |z_1|^2)(1 - |z_2|^2)} \right)
\]

        An expression for the heat kernel for general $d$ on $\mathcal{H}^d$ is given in \cite{grigoryan_hyperbolic}. For $d=2$, we have 
\begin{equation}
\label{eq:transition_density_hyperbolic}
    p(s,x;t,y) = \frac{\sqrt{2}e^{-(t-s)/2}}{(2\pi(t-s))^{3/2}}\int_{\rho_x}^\infty \frac{u\exp\left(-\frac{u^2}{2(t-s)}\right)}{\sqrt{\cosh u - \cosh\rho_x}}\dd u ,
\end{equation}
where $\rho_x = \dist(x,y)$ denotes the geodesic distance between $x$ and $y$. 

Despite the manifold  not being compact, we apply our methodology and construct the guided process.  Details on evaluating $\nabla_x \log p(s,x; t,y)$ are given in \Cref{app:heat_kernel_H2}. 
We consider three vector fields
\begin{itemize}
    \item $V_1(x) = -20x$ which strongly direct paths towards the center of the disc;
    \item $V_2(x)=x$ which sends paths towards the boundary of the disc;
\item $V_3(x) = [5 (1-x_1^2-x_2^2)^2; 0]$ which sends paths in the east direction. A visualisation of this vector field is in the left panel of \Cref{fig:V3omdraaien}. 
\end{itemize}
For each of these vector fields, we simulate $4$ independent realisations of the guided process on $[0,1]$ where each path starts at $x_0=[0; 0.6]$ and ends at $x_1=[0; -06]$. The paths were simulated on a grid with mesh-width $0.001$ that was then mapped to a non-equidistant grid through the map $s\mapsto s(2-s)$ as advocated in Section 5 of \cite{meulen_schauer_bayest}.  The simulated paths are shown in \Cref{fig:hyperbolic-bridges}. One can see in the top panel that the process is pushed rather quickly towards the center, stays most of the time near the center, and then moves to $x_1$. The bottom panel indeed show the paths tend to go eastwards. 
\begin{figure}[h]
\begin{center}
    \includegraphics[scale=0.8]{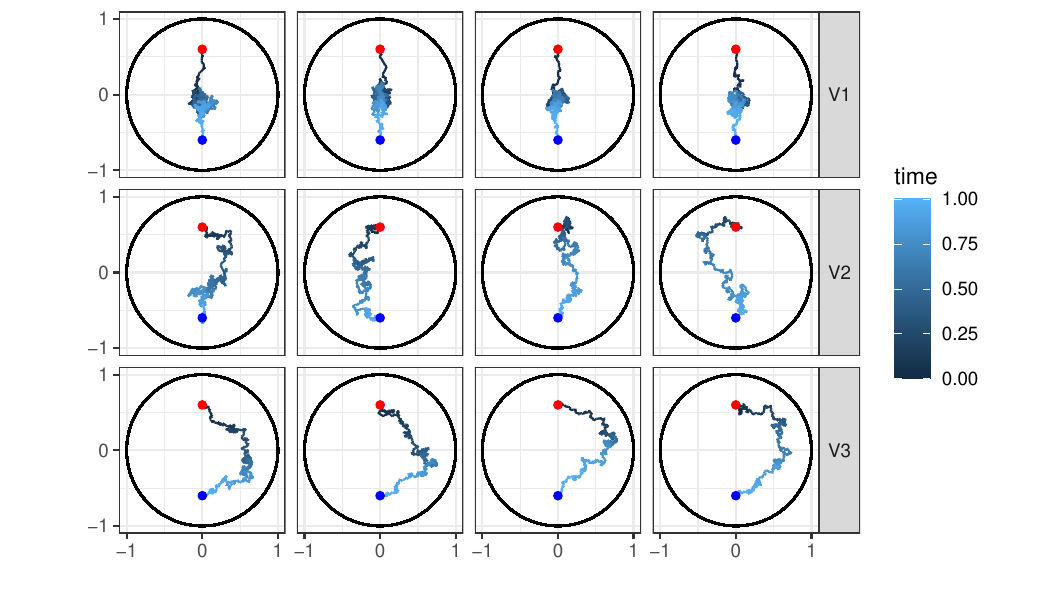}
    \caption{Each panel shows an independent realisation of the guided process connecting $x_0$ (red dot) and $x_1$ (blue dot). Top row: vector field $V_1$, middle row: vector field $V_2$; bottom row: vector field $V_3$. \label{fig:hyperbolic-bridges}}
\end{center}
\end{figure}
\begin{figure}[h]
\begin{center}
\includegraphics[scale=0.8]{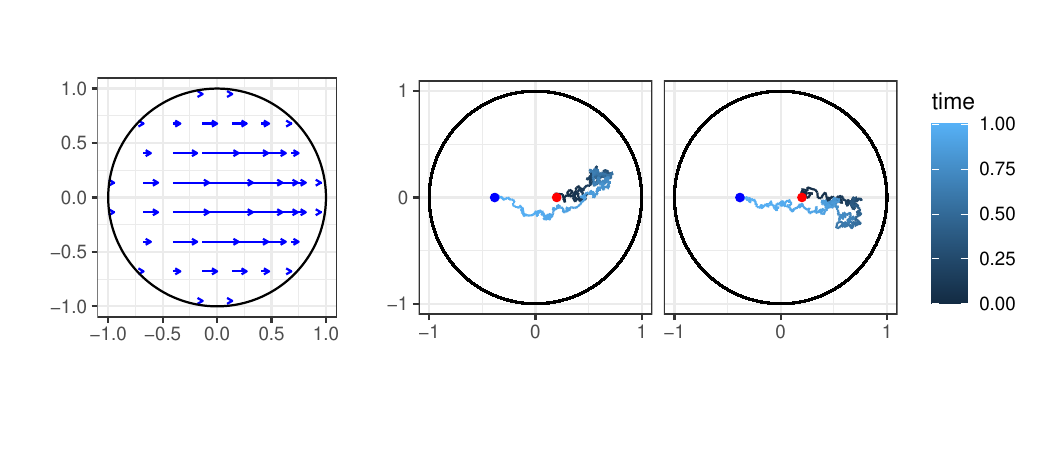}
    \caption{Left: visualisation of vector field $V_3$. Right: each panel shows an independent realisation of the guided process connecting $x_0=[0; 0.2]$ (red dot) and $x_1=[-0.8;0]$ (blue dot). \label{fig:V3omdraaien}}
\end{center}
\end{figure}
In the right panel of  \Cref{fig:V3omdraaien} we show two  simulated guided paths where we changed the starting point to $[0; 0.2]$ and condition the process to hit $x_1=[-0.8;0]$ at time $1$. One can see that paths initially tend to move eastwards due to the vector field  $V_3$ but are subsequently forced to hit $x_1$ at time $1$.

For vector field $V_3$ we run the pCN-scheme of \Cref{alg:crank-nicolson} with $\lambda=0.4$ for $500$ iterations and show every $50$th iterate in \Cref{fig:hyperbolic-pcn}. The acceptance percentage was approximately $80\%$, showing that the guided process qualitatively approximates the true conditioned process quite well.

\begin{figure}[h]
    \centering
    \includegraphics[scale=0.7]{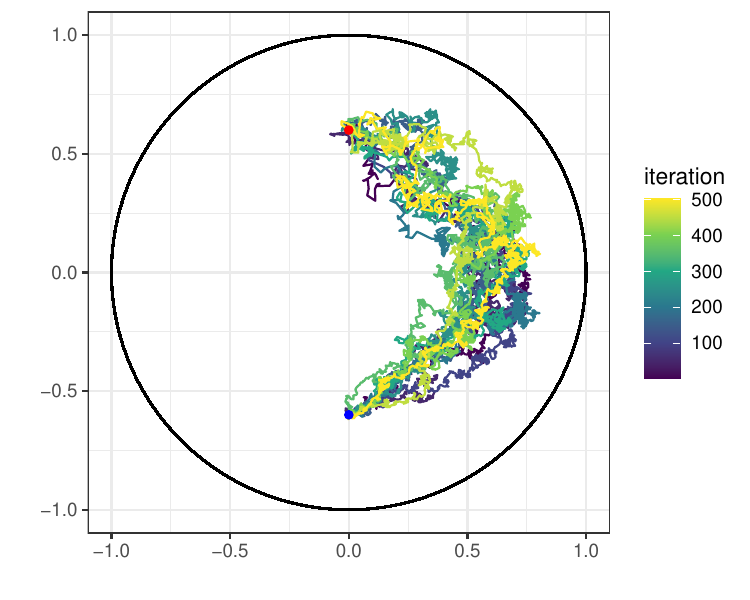}
        
    \caption{Every $50$th iteration out of $500$ iterations after running the pCN-scheme of \Cref{alg:crank-nicolson} for vector field $V_3$ on $\mathcal{H}^2$.
    \label{fig:hyperbolic-pcn}}
\end{figure}

Lastly, we apply the Gibbs sampler described \Cref{subsec:parameter_estimation} to estimate the parameter $\theta$ appearing in the vector field $V_\theta(x) = [\theta(1-x_1^2-x_2^2)^2 ; 0]$ with a $N(0,20^2)$-prior. \Cref{fig:gibbs_hyperbolic} demonstrates the estimates of 1000 iterations with tuning parameter $\lambda = 0.9$. The observations were observations of a forward simulated process with $\theta=5$ on an equally spaced grid of 100 points in the time interval $[0,1]$.

\begin{figure}[h]
\centering
\includegraphics[width = 0.6\textwidth]{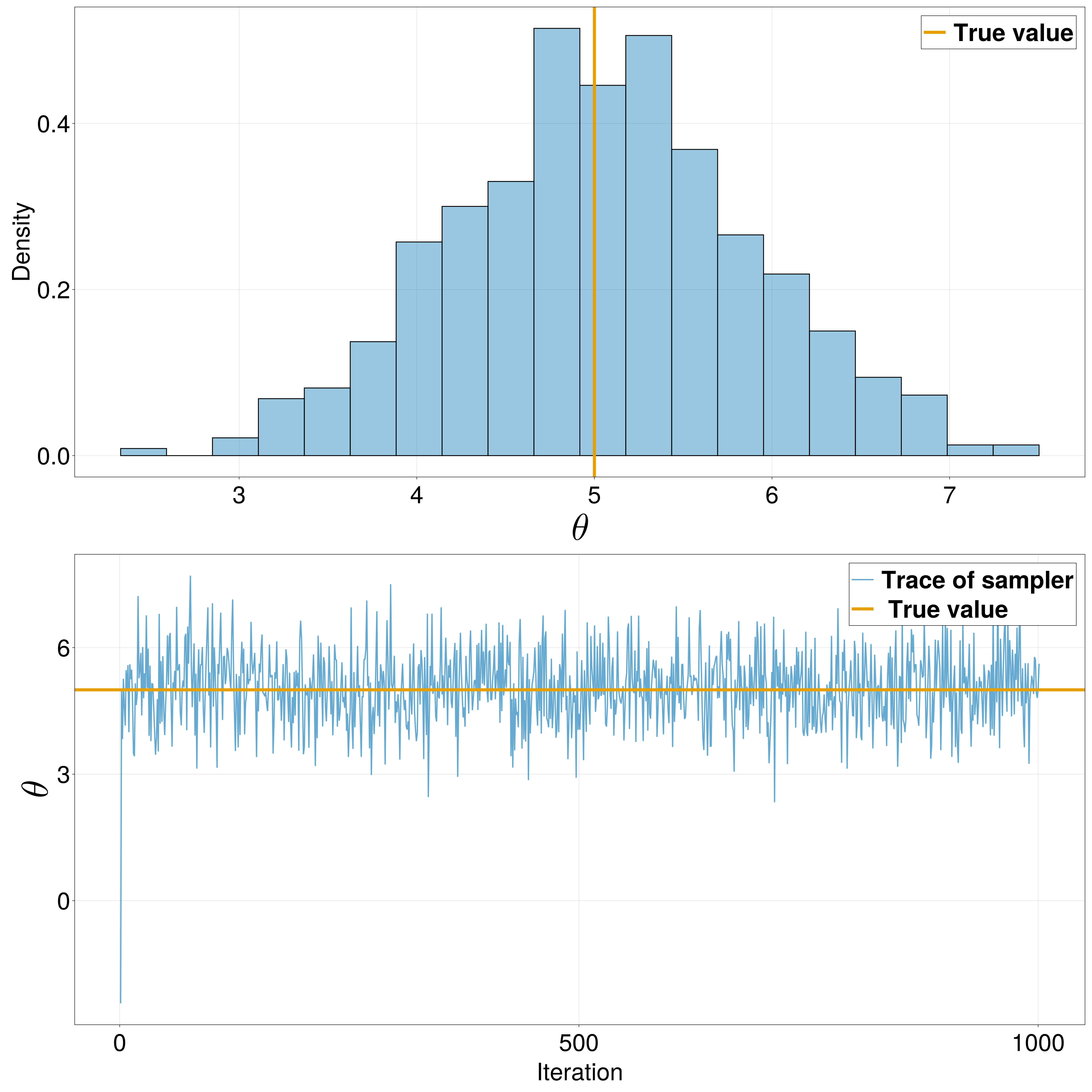}
\caption{Top: Histogram of 900 from the posterior distribution of the parameter $\theta$ appearing in the vector field  $V_\theta(x) = [\theta(1-x_1^2-x_2^2)^2 ; 0]$ after omitting 100 burn-in samples. Below: Trace of the sampler. The samples are based on 100 observations of a forward simulated process with $\theta = 5$}
\label{fig:gibbs_hyperbolic}
\end{figure}

These simulations suggest that the results of this article may also hold on the Poincar\'e disk, which is an example of a noncompact manifold.

 \section{Other types of conditionings}
 \label{sec:extensions}

Up to this point, we have considered the problem of diffusion bridge simulation, where $h(t,x)=p(t,x; T, x_T)$. We now consider variants of this problem.

First consider the following model
\begin{align*}
    (X_t,\, t\in [0,T]) & \sim \bb{P} \\
    V\mid X_T=x_T &\sim q_{V\mid X_T}(\cdot \mid x_T)
\end{align*}
for a conditional density $q_{V\mid X_T}$. This model corresponds to observing a realisation $v$ from  the conditional distribution given $V\mid X_T=x_T$. Suppose interest lies in the distribution of $X_T \mid V=v$. 
If we let the measure $\bb{P}^h$ be the measure induced by \begin{equation}\label{eq:hfrommu}
    h(t,x) = \int p(t,x;T,z) q_{V\mid X_T}(v\mid z)\dd \mathrm{Vol}_M(z) .
 \end{equation}  then $\mathfrak A h= 0$ and $h \in \mathcal H$.  
 
As in Example 2.4 in \cite{corstanje2021conditioning}, for  test functions $f\colon M \to \mathbb R$ 
  \[ \bb{E}^h  f(X_t) =\int_M \bb{E}\left( f(X_t)\mid X_T=z\right) q_{X_T\mid V}(z\mid v)\dd\mathrm{Vol}_M(z) , \]
		where
		\[q_{X_T\mid V}(z\mid v) = \frac{p(0,x_0;T,z)q_{V\mid X_T}(v\mid z) }{h(0,x_0)}\]
is the conditional density of $X_T\mid V=v$ with respect to the volume measure. Hence, sampling under $\bb{P}^h$ is equivalent to first sampling $x_T$ from the distribution of $X_T \mid V=v$ and subsequently sampling a bridge connecting $x_0$ at time $0$ and $z$ at time $T$.

In a variant, one can take any density function $q(z)$ with respect to the volume measure to obtain a transformation as follows:
\begin{equation} \label{eq:hforcing}
h(t,x) = \int \frac{p(t,x;T,z)q(z)}{p(0,x_0;T,z)}\dd\mathrm{Vol}_M(z). 
\end{equation}
This corresponds to a conditioning where $X$ is started at $X_0 = x_0$ and forced into the marginal distribution with density $q$ at time $T$, such as seen in \cite{baudoin2002}. This is important for generative diffusion models on manifolds.

\section{Proofs}
\label{sec:proofs}

The proofs in Sections \ref{subsec:proofinvariance2} and \ref{app:proof-of-mainthm} are derived from the following theorem, which  is a slight reformulation of Theorem 3.3 of \cite{corstanje2021conditioning}. 
		\begin{thm}
		\label{thm:absolute-continuity} Let $g, h \in \mathcal{H}$.
		Suppose there exists a $t_0 \in [0,T)$ and a family of $\mathcal{F}_t$-measurable events $\{A_k^t\}_k$ for all $t\in[t_0,T)$ such that 
		\begin{enumerate}[label={\color{blue} (\ref{thm:absolute-continuity}\alph*)}]	
		\item \label{ass:main-events}
			For all $k$ and $t_0\leq s\leq t< T$, $A_k^t \subseteq A_{k+1}^t$ and $A_k^t\subseteq A_k^s$. We denote the limits by 
			\[ A^t = \bigcup_k A_k^t ,\qquad A_k^T = \bigcap_{t<T}A_k^t,\qquad A^T = \bigcap_{t<T}A^t \]
		\item \label{ass:main-expectations}
			$\bb{E}\left(g(t,X_t)\mid \mathcal{F}_s\right)$ is $\bb{P}$-almost surely bounded uniformly in $t$ and for all $s\in[t_0,T)$ and $x\in M$,
			\[ \lim_{t\uparrow T} \bb{E}\left(g(t,X_t)\mid X_s=x\right) = h(s,x) \]
		\item \label{ass:main-fraction}
			\[ \lim_{k\to\infty}\lim_{t\uparrow T} \bb{E}^h \left(\frac{g(t,X_t)}{h(t,X_t)}\ind_{A_k^t}\right) = 1 \]
		\item \label{ass:main-likelihood}
			For all $k$, $\Psi_{T,g}(X)\ind_{A_k^t}$, with $\Psi_{T,g}(X)$ defined in \eqref{eq:Psi},
			is $\bb{P}^g$-almost surely uniformly bounded in $t$ on $[t_0,T)$. 
        \item \label{ass:P(AT)=1}
        $\bb{P}^g(A^T)=1$. 
		\end{enumerate}
		Then $\bb{P}_T^h \sim \bb{P}_T^g$ with
		\begin{equation}
			\label{eq:radon-nikodym-closed}
				\der{\bb{P}^h_T}{\bb{P}^g_T} = \frac{g(0,X_0)}{h(0,X_0)}\exp\left(\int_0^T \frac{\scr{A}g(s,X_s)}{g(s,X_s)}\dd s\right)
			\end{equation}
		\end{thm}
	
\subsection{Proof of Theorem \ref{thm:invar}}\label{subsec:proofinvariance1}
	Suppose $X$ has infinitesimal generator $\left(\scr{L}_X, \scr{D}\left(\scr{L}_X\right)\right)$.
Let $f\colon N\to \bb{R}$ be such that $f\circ \Phi \in \scr{D}\left(\mathfrak{L}_X\right)$. Then the infinitesimal generator $\scr{L}_Y$ of $Y$ satisfies
		\begin{equation}
		\label{eq:invariance-generators}
			\begin{aligned}
				\scr{L}_Yf(y) &= \lim_{t\downarrow 0} \frac{\bb{E}\left(f(Y_t)\mid Y_0=y\right)-f(y)}{t} \\
				  &=  \lim_{t\downarrow 0} \frac{\bb{E} \left( (f\circ\Phi)(X_t)\mid X_0 = \Phi^{-1}(y)\right) - (f\circ\Phi)\left(\Phi^{-1}(y)\right)}{t} \\
				  &= \scr{L}_X \left(f\circ\Phi\right) \left(\Phi^{-1}(y)\right),
			\end{aligned}
		\end{equation}
For $f\colon [0,T] \times N \to \bb{R}$, let the pullback $\Phi^*f \colon [0,T] \times M\to \bb{R}$ given by  $(\Phi^* f)(t,x) = f(t,\Phi(x))$ be in the domain of $\scr{A}_X$. Then we have 
\[ \mathfrak{A}_Y f(t,y)= \left(\scr{A}_X (\Phi^* f)\right)\left(t,\Phi^{-1}(y)\right), \quad t\in [0,T],\quad y\in N.
	\]
Take	$f=\Phi_* h$, then $\Phi^* f(t,x)=(\Phi_* h)(t,\Phi(x))=h(t,x)$ is in the domain of $\scr{A}_X$, by assumption. The preceding display therefore gives
\begin{equation}\label{eq:trans_generators}
	(\scr{A}_Y \Phi_* h)(t,y) = (\scr{A}_X h)(t,\Phi^{-1}(y)) =\left(\Phi_* (\scr{A}_X h)\right)(y),
\end{equation}
	
where the second equality follows from the definition of $\Phi_* h$. 

We have (cf.\ \Cref{eq:def-Mtf-Etf})
\[ E_t^{\Phi_* h}(Y) = 	\frac{\Phi_* h(t,Y_t)}{\Phi_* h(0,Y_0)}\exp\left(-\int_0^t \frac{\scr{A}_Y(\Phi_* h)(s,Y_s)}{\Phi_* h(s,Y_s)}\dd s\right) 
\]
By \Cref{eq:trans_generators}, $\scr{A}_Y(\Phi_* h)(s,Y_s)=(\scr{A}_X h)(s,\Phi^{-1}(Y_s))= (\scr{A}_X h)(s,X_s)$. By definition of $\Phi_* h$ we have $\Phi_* h(t,Y_t)= h(t,X_t)$. Substituting these two results in the previous display we obtain $ E_t^{\Phi_* h}(Y) = E_t^h(X)$. 
This implies
\[ \mathbb{E}  \psi\left(\Phi(X^h)\right) = \mathbb{E} \left[ \psi(\Phi(X)) E_t^h(X)\right] =  \mathbb{E}\left[ \psi(Y) E_t^{\Phi_* h}(Y)\right].\]

\subsection{Proof of \Cref{thm:invariance-of-GPs}}\label{subsec:proofinvariance2}

We start with a lemma used in the proof.
\begin{lem}
		\label{prop:InvarianceEquivalenceProp1}
		$Y$ admits a transition density $p_Y(s,y;t,z) = p(s, \Phi^{-1}(y);t,\Phi^{-1}(z))$ with respect to the measure on $N$ obtained from the pushforward Riemannian volume measure on $M$ through $\Phi$, $\Phi_*\left(\mathrm{Vol}_M\right)$.

\end{lem}

	\begin{proof}
		For $A\subseteq N$ measurable, $y\in N$ and $s\leq t$, we have
	\begin{align*} 
		\bb{P}(Y_t\in A\mid Y_s=y) &= \bb{P}\left( X_t \in \Phi^{-1}(A) \mid X_s=\Phi^{-1}(y) \right) \\
		&= \int_{\Phi^{-1}(A)} p\left(s, \Phi^{-1}(y); t,u\right)\dd \mathrm{Vol}_M(u)  \\
		&= \int_A p(s, \Phi^{-1}(y) ; t, \Phi^{-1}(z)) \dd \Phi_*\left(\mathrm{Vol}_M\right)(z),
		\end{align*}
  where the last equality follows from the substitution rule, see e.g. Proposition 4.4.12 of \cite{Lovett2010}
	\end{proof}

\begin{proof}[Proof of \Cref{thm:invariance-of-GPs}]
As \ref{ass:main-events} is satasfied by default, we check assumptions \ref{ass:main-expectations}, \ref{ass:main-fraction} \ref{ass:main-likelihood} and \ref{ass:P(AT)=1} for $Y$, $\Phi_* g$, $\Phi_* h$ and events $\{A^t_k\}_k$. 

Assumption \ref{ass:main-expectations} is satisfied since
	\begin{align*}
		\bb{E}\left(\Phi_* g(t,Y_t)\mid Y_s=y\right) &=  \int g(t, \Phi^{-1}(z))p_Y(s,y;T,z)\dd\Phi_*\left(\mathrm{Vol}_M\right)(z)\\
		&= \int g(t, \Phi^{-1}(z))p(s, \Phi^{-1}(y);T,\Phi^{-1}(z)) \dd\Phi_*\left(\mathrm{Vol}_M\right)(z)\\
		&= \int g(t, x)p(s, \Phi^{-1}(y);T, x) \dd\mathrm{Vol}_M(x)\\ &  \to h(s, \Phi^{-1}(y)) = \Phi_* h(s,y),\quad \text{as}\:  t\uparrow T
	\end{align*}

	Assumption \ref{ass:main-fraction} is satisfied  since 
\[	\lim_{k\to\infty}\lim_{t\uparrow T} \bb{E}^{\Phi_* h}\left( \frac{\Phi_* g(t, Y_t) }{\Phi_*{h}(t, Y_t)}\ind_{A_k^t} \right) = \lim_{k\to\infty}\lim_{t\uparrow T} \bb{E}^h\left( \frac{g(t, X_t)}{h(t, X_t)}\ind_{A_k^t} \right) = 1. \]
	Here, the first equality follows from $E_t^{\Phi_* h}(Y) = E_t^h(X)$, which was derived in the proof of \Cref{thm:invar}.  The final equality follows from Assumption \ref{ass:main-fraction}.

To see that Assumption \ref{ass:main-likelihood} is satisfied, note that from the proof of \Cref{thm:invar}, we obtain  $\scr{A}_Y(\Phi_* g)(s,Y_s)= (\scr{A}_X g)(s,X_s)$
	and thus 
	\[ \exp\left(\int_0^t \frac{\mathfrak{A}_Y\Phi_* g(s,Y_s)}{\Phi_* g(s, Y_s)} \dd s \right) = \exp\left( \int_0^t \frac{\mathfrak{A}_X g(s,X_s)}{g(s, X_s)}\dd s \right) ,\] which verifies the assumption. 

For \ref{ass:P(AT)=1}, we note that $E_t^g(X) = E_t^{\Phi_* g}(\Phi(X))$ and thus the measures $\bb{Q}^{\Phi_* g}$ and $\bb{P}^g$ define the same probability $1$ sets. 
\end{proof}

\subsection{Proof of \Cref{thm:absolute-continuity-drift}}
        \label{app:proof-of-mainthm}

       We also need the following two results where we assume $M$ to be a compact manifold.

        \begin{thm}[Theorem 5.3.4 of \cite{Hsu2002}]
		\label{thm:bounds-on-heat-kernel}
			For all $t \in(0,T)$ and $x,y\in M$,
			\[ \frac{C_1}{t^{d/2}} \exp\left(-\frac{\dist(x,y)^2}{2t}\right) \leq \kappa (t,x,y) \leq \frac{C_2}{t^{(2d-1)/2}} \exp\left(-\frac{\dist(x,y)^2}{2t}\right) \]
		\end{thm}

        \begin{thm}[Theorem 5.5.3 of \cite{Hsu2002}]
        \label{thm:upper-bound-gradlogg}
            There is a constant $C$ such that for all $t\in(0,T)$ and $x,y\in  M$, 
            \[ \abs{\nabla\log\kappa(t,x,y)} \leq C\left(\frac{\dist(x, y)}{t} + \frac{1}{\sqrt{t}}\right)\]
        \end{thm}
        
        \begin{proof}[Proof of \Cref{thm:absolute-continuity-drift}]
            Let $\eta(t) = (T-t)^\eps$ for $\eps \in(0,1/2)$ and define the function 
            \begin{equation}
                \label{eq:defzeta}
                \zeta(t,x) = -\eta(t)\log g(t,x),
            \end{equation}
            We recall that $g(t,x) = \kappa(T-t, x, x_T)$. It follows from \Cref{thm:bounds-on-heat-kernel} that 
            \begin{equation}
            \label{eq:zeta-bounds}
                \beta_1(t)+\eta(t)\frac{\dist(x,x_T)^2}{2(T-t)} \leq \zeta(t,x) \leq \beta_2(t) + \eta(t)\frac{\dist(x,x_T)^2}{2(T-t)}        
            \end{equation}
            where 
            \begin{equation}
            \label{eq:betas}
                \begin{aligned}
                    \beta_1(t) &= -\eta(t) \log\left( \frac{C_2}{(T-t)^{(2d-1)/2}}\right) \\
                    \beta_2(t) &= -\eta(t) \log\left( \frac{C_1}{(T-t)^{d/2}}\right)
                \end{aligned}
            \end{equation}
            Now let $t_0\in[0,T)$ be such that $\beta_2(t)\geq \beta_1(t) \geq 0$ for all $t\in[t_0,T)$ and observe that $\eta$ and $t_0$ are such that
            \begin{enumerate}[label={\color{blue} (\textit{\roman*})}]	
                \item \label{obs:zeta-positive} $\eta(t)\geq 0$ and $\zeta(t,x)\geq 0$ for all $t\in[t_0,T)$
                \item \label{obs:integrable} $t\mapsto \frac{\eta(t)}{T-t}$ is integrable on $[t_0,T)$
                \item \label{obs:logbounded} $t\mapsto \eta(t)\log\left(\frac{1}{T-t}\right)$ is bounded on $[t_0,T)$
            \end{enumerate}
            Define the collection of events $A_k^t$ as 
            \begin{equation}
                \label{eq:defAkt}
                A_k^t = \left\{ \sup_{t_0\leq s<t} \zeta(s,X_s) \leq k \right\}, \qquad t\in[t_0,T), k=1,2,\dots.
            \end{equation}

        Then \ref{ass:main-events} is satisfied by construction. By Theorem 4.1.4 (3) of \cite{Hsu2002},
	\begin{equation}
	\label{eq:proof-of-guided-limit}
		\begin{aligned}
			\lim_{t\uparrow T} \bb{E}\left(g(t,X_t) \mid X_s=x\right) &= \lim_{t\uparrow T} \int \kappa({T-t},y, x_T)p(s,x;t,y) \dd \mathrm{Vol}_M(y) 
			= p(s,x;T,x_T) = h(s,x) 
		\end{aligned}
	\end{equation}
	and thus \ref{ass:main-expectations} is satisfied as well. For \ref{ass:main-fraction}, we first note that 
	\[ \bb{E}^h \left( \frac{g(t,X_t)}{h(t,X_t)} \ind_{A_k^t} \right) = \bb{E}^h \left(\frac{g(t,X_t)}{h(t,X_t)}\right) - \bb{E}^h \left( \frac{g(t,X_t)}{h(t,X_t)}\ind_{(A_k^t)^c}\right) \]
	By \Cref{lem:limitfraction}, the first term tends to 1 as $t\uparrow T$. Therefore it suffices to show that the second term tends to $0$ upon taking $\lim_{k\to\infty}\lim_{t\uparrow T}$ to prove that \ref{ass:main-fraction} is satisfied.  To show this, we consider the sequence of stopping times $\sigma_k = T\wedge \inf\{t\geq 0 : \zeta(t,X_t) > k\}$ and observe that $(A_k^t)^c = \{t>\sigma_k\}$. Now 
	\begin{align*}
		\bb{E}^h \left( \frac{g(t,X_t)}{h(t,X_t)}\ind_{\{t>\sigma_k\}}\right) &= \bb{E}\left(\frac{g(t,X_t)}{h(0,x_0)}\ind_{\{t>\sigma_k\}}\right) \\ 
		&= \bb{E}\left( \frac{ \ind_{ \{t > \sigma_k\} } }{h(0,x_0)} \bb{E}\left( g(t,X_t) \mid\mathcal{F}_{\sigma_k}\right)\right) \\
		&= \bb{E}\left(	 \frac{ \ind_{ \{t > \sigma_k\} } }{h(0,x_0)} \int_M p(\sigma_k, X_{\sigma_k} ; t, z)g(t,z)\dd z \right) 
	\end{align*}
	Here, the first equality follows from applying the change of measure \eqref{eq:change-of-measure} with $\scr{A}h=0$. By \Cref{ass:transitiondensity}, we obtain an upper bound given by 
	\[ \bb{E}\left(\frac{\ind_{\{t>\sigma_k\}}}{h(0,x_0)}\gamma_1e^{-\gamma_3(t-\sigma_k)}\int_M \kappa(\gamma_2(t-\sigma_k), X_{\sigma_k},z)\kappa(T-t, z,x_T)\dd \mathrm{Vol}_M(z)   \right) \]
	It follows from the Chapman-Kolmogorov relations, see e.g. Theorem 4.1.4 (5) of \cite{Hsu2002}, that this is equal to 
	\[ \bb{E}\left( \frac{\ind_{\{t>\sigma_k\}}}{h(0,x_0)}\gamma_1e^{-\gamma_3(t-\sigma_k)} \kappa\left(\gamma_2(t-\sigma_k)+T-t, X_{\sigma_k}, x_T)\right) \right). \]
	By \Cref{thm:bounds-on-heat-kernel}, we can find a positive constant $\gamma_4$ such that this expression is bounded by  
	\[\begin{gathered} \bb{E}\left( \frac{\ind_{\{t>\sigma_k\}}}{h(0,x_0)}\gamma_1e^{-\gamma_3(t-\sigma_k)} \gamma_4\left(\gamma_2(t-\sigma_k)+T-t\right)^{-(2d-1)/2}\right.  \\ \left.
  \times \exp\left(-\frac{\dist(X_{\sigma_k},x_T)^2}{2\gamma_2(t-\sigma_k)+2(T-t)}\right) \right) 
        \end{gathered}\]
	Now observe that 
	\begin{enumerate} 
		\item Since $X$ is continuous in time, $\zeta\left(\sigma_k, X_{\sigma_k}\right) = k$. It thus follows from \eqref{eq:zeta-bounds} that 
        \[\frac{k-\beta_2(\sigma_k)}{\eta(\sigma_k)}\leq \frac{\dist(X_{\sigma_k},x_T)^2}{2(T-\sigma_k)} \leq \frac{k-\beta_1(\sigma_k)}{\eta(\sigma_k)} \]
		\item $e^{-\gamma_3(t-\sigma_k)} \leq 1$ on $\{t>\sigma_k\}$ 
		\item  Let $\underline{\gamma}=\gamma_2\wedge 1$ and $\bar{\gamma}=\gamma_2\vee 1$. Then $\underline{\gamma}(T-\sigma_k) \leq \gamma_2(t-\sigma_k)+T-t\leq \bar{\gamma}(T-\sigma_k)$. 
	\end{enumerate}
    It now follows from the preceding, combined with observations 1--3 that a positive constant $\gamma_5$ exists such that 
    \begin{equation*}
    \begin{aligned}
        \bb{E}^h\left( \frac{g(t,X_t)}{h(t,X_t)}\ind_{\{ t>\sigma_k \}}\right) &\leq \bb{E}\left(\gamma_5\ind_{\{t>\sigma_k\}} (T-\sigma_k)^{-(2d-1)/2} \exp\left(- \frac{k-\beta_2(\sigma_k)}{\eta(\sigma_k)\bar{\gamma}}\right)\right) \\
        &= \bb{E}\left( \gamma_5 C_1^{-1/\bar\gamma} \ind_{\{t>\sigma_k\}} (T-\sigma_k)^{ -\frac{2d-1}{2}-\frac{d}{2\bar\gamma}} \exp\left( -\frac{k}{\eta(\sigma_k)\bar\gamma}\right)\right)  
    \end{aligned}    
    \end{equation*}
    We now substitute $\eta(t) = (T-t)^\eps$. Upon noting that trivially $\tau := T-\sigma_k \in [0,T]$, we deduce that 
    \[ \bb{E}^h \left( \frac{g(t,X_t)}{h(t,X_t)}\ind_{\{t>\sigma_k\}}\right) \leq \gamma_5C_1^{-1/\bar\gamma}  \sup_{0\leq \tau < \infty } \tau^{-\gamma_6}\exp\left(-\frac{k}{\bar\gamma}\tau^{-\eps}\right),  \]
    where $\gamma_6 = (2d-1)/2+d/(2\bar\gamma)$. The function of $\tau$ over which the $\sup$ is taken is concave and attains its maximum at $\left(\gamma_6 \bar\gamma/(k\eps)\right)^{-1/\eps}$. Hence, 
    \begin{equation}\label{eq:proof-absolute-continuity-expectationto0} 
    \bb{E}^h \left( \frac{g(t,X_t)}{h(t,X_t)}\ind_{\{t>\sigma_k\}}\right) \leq \gamma_5C_1^{-1/\bar\gamma}\left(\frac{\eps}{\gamma_6\bar\gamma}\right)^{-\gamma_6/\eps} k^{-\gamma_6/\eps}\exp\left(-\frac{\gamma_6}{\eps}\right) 
    \end{equation}
	Since, \eqref{eq:proof-absolute-continuity-expectationto0} tends to $0$ as $k\to\infty$, \ref{ass:main-fraction} is satisfied.

	For \ref{ass:main-likelihood}, note that as $\partial_t g + \frac12\Delta g = 0$, $\scr{A}g = \left(\scr{L}-\frac12 \Delta\right)g$. Hence, $\scr{A}g/g = V(\log g)$. Clearly, by the Cauchy-Schwarz inequality
	\[ \abs{ V(\log g)(t,x) } \leq \abs{V(x)}\abs{\nabla\log g} \]
	Since $V$ is smooth and $M$ is compact, $\abs{V}$ is bounded. Hence, by \Cref{thm:upper-bound-gradlogg}, there exists a constant $C$ such that 
	\[ \abs{V(\log g)(t,x)} \leq C\left(\frac{\dist(x,x_T)}{T-t} + \frac{1}{\sqrt{T-t}}\right) \] 
    Now notice that, by \eqref{eq:zeta-bounds}, on $A_k^t$, we have that 
    \begin{equation}
        \label{eq:geodesic-distance-on-Akt}
        \dist(X_t, x_T) \leq \sqrt{2(T-t)^{1-\eps}\left(k-\beta_1(t)\right)}
    \end{equation}
    Hence, we obtain that 
    \[ \abs{V(\log g)(t, X_t)}\ind_{A_k^t} \leq C \left( \frac{\sqrt{2(k-\beta_1(t))}}{\sqrt{(T-t)^{1+\eps}}} + \frac{1}{\sqrt{T-t}} \right)\]
	Since $\beta_1(t)$ is bounded on $[t_0,T)$ (by \ref{obs:logbounded}) and tends to  $0$ as $t\uparrow T$, the above term is integrable on $[t_0,T]$, \ref{ass:main-likelihood} is satisfied. We proof \ref{ass:P(AT)=1} separately in \Cref{thm:probability1}. The result now follows from \Cref{thm:absolute-continuity}, where the form of the Radon-Nikodym derivative in \Cref{eq:radon-nikodym-closed} follows from the observation that $\scr{A}g/g = V(\log g)$.
 
        \end{proof}

	\begin{lem}
	\label{lem:limitfraction}
		Suppose \ref{ass:main-expectations} is satisfied and let $s<T$. For any bounded $\mathcal{F}_s$-measurable function $f_s$, 
		\[\lim_{t\uparrow T} \bb{E}^h \left(f_s(X)\frac{g(t,X_t)}{h(t,X_t)}\right) = \bb{E}^h f_s(X) \]
	\end{lem}
	\begin{proof} Since, $\dd\bb{P}_t^h = h(t,X_t)/h(0,x_0) \dd\bb{P}_t$,
		\begin{equation}
		\begin{aligned}
			\lim_{t\uparrow T} \bb{E}^h \left(f_s(X)\frac{g(t,X_t)}{h(t,X_t)}\right) &= \lim_{t\uparrow T} \bb{E} \left(\frac{f_s(X)}{h(0,x_0)}g(t,X_t)\right) \\
			&= \lim_{t\uparrow T} \bb{E}\left( \frac{f_s(X)}{h(0,x_0)}\bb{E}\left(g(t,X_t)\mid \mathcal{F}_s\right)\right) = \bb{E}\left(f_s(X)\frac{h(s,X_s)}{h(0,x_0)} \right) = \bb{E}^hf_s(X)
		\end{aligned}
		\end{equation}
	Here, the second but last equality follows from \ref{ass:main-expectations} and dominated convergence. 
	\end{proof}

\begin{thm}
	\label{thm:probability1}
		Let $A_k^t$ be defined as in \eqref{eq:defAkt}. Then $\bb{P}^g(A^T)=1$.
	\end{thm}

\begin{proof}
    To reduce notation, we write a subscript $t$ to indicate evaluation in $(t,X_t)$ for functions of $(t,x)\in[t_0,T)\times M$. It is sufficient to show that $\zeta_t$ is $\bb{P}^g$-almost surely bounded as $t\uparrow T$. To do this, we find an upper bound on 
    \begin{equation}
        \label{eq:defxi}
        \xi_t := \zeta_t-\beta_1(t),
    \end{equation} which in turn yields the result as $\beta_1(t)\downarrow 0$ as $t\uparrow T$. 
    
    Notice that, under $\bb{P}^g$, by It\^o's formula,
    \begin{equation}
    \label{eq:xi-ito}
        \xi_t = \xi_{t_0} + M_t^\xi + \int_{t_0}^t \scr{A}^g\xi_s \dd s,
    \end{equation}
    where $M^\xi$ denotes the process $M^\xi_t = \int_{t_0}^t \eta(s)\nabla \log g_s \dd W_s $. 

    The remainder of the proof is structured as follows: First we find an upper bound of $M_t^\xi + \int_{t_0}^t \scr{A}^g \xi_s\dd s$ in terms of functions of $t$ and $\xi$. Then we apply a generalized Gr{\"o}nwall inequality to find an upper bound for $\xi_t$, which is bounded as $t\uparrow T$. 

    First observe that 
    \begin{equation}
    \label{eq:Mtxi}
        M_t^\xi =  M_t^\xi - \frac12 \left[M^\xi\right]_t + \frac12 \int_{t_0}^t \abs{\eta(s)\nabla\log g_s}^2 \dd s
    \end{equation}
    By Lemma A.2 of \cite{corstanje2021conditioning}, there exists an almost surely finite random variable $C$ such that 
    \begin{equation} 
        \label{eq:exponential-martingale-bound}
            M_t^\xi - \frac12 \left[M^\xi\right]_t \leq C\log\left(\frac{1}{T-t}\right).
    \end{equation}
    Moreover, we infer from \eqref{eq:zeta-bounds} that $\dist(X_t,x_T) \leq \sqrt{{2(T-t)\xi_t}/{\eta(t)}}$.
    Hence, by \Cref{thm:upper-bound-gradlogg}, there is a positive constant $c_1$ such that
    \begin{equation}
        \label{eq:gradlogg-bound}
        \abs{\nabla\log g_t}\leq  c_1 \left( \sqrt{\frac{2\xi_t}{\eta(t)(T-t)}} + \frac{1}{\sqrt{T-t}}\right)
    \end{equation}
    Therefore, by combining \eqref{eq:Mtxi} with \eqref{eq:exponential-martingale-bound} and \eqref{eq:gradlogg-bound}, we find a bound on $M^\xi_t$ in terms of $t$ and $\xi$ through 
    \begin{equation}
        \label{eq:Mtxi-bound}
        M_t^\xi \leq C\log\left(\frac{1}{T-t}\right) + \frac{c_1^2}{2}\int_{t_0}^t \eta(s)^2\left( \sqrt{\frac{2\xi_s}{\eta(s)(T-s)}} + \frac{1}{\sqrt{T-s}}\right)^2 \dd s
    \end{equation}
    
    Next, we study $\scr{A}^g \xi$. A direct computation yields that  
    \begin{equation}
    \label{eq:Agxi}
    \begin{aligned}
        \scr{A}^g\xi_t &= \scr{A}^g\zeta_t -\partial_t \beta_1 \\
        &= \scr{A}^g\zeta_t  + \log\left(\frac{C_2}{(T-t)^{(2d-1)/2}}\right)\partial_t\eta(t) + \eta(t) \partial_t\log\left(\frac{C_2}{(T-t)^{(2d-1)/2}}\right) \\
        &= \scr{A}^g\zeta_t  + \eta(t)\log\left(\frac{C_2}{(T-t)^{(2d-1)/2}}\right)\partial_t\log\eta(t) +\eta(t)\frac{2d-1}{2(T-t)} \\
        &= \scr{A}^g\zeta_t - \beta_1(t)\partial_t\log\eta(t) + \eta(t)\frac{2d-1}{2(T-t)} 
    \end{aligned}
    \end{equation}

    In order to compute $\scr{A}^g\zeta_t$, we use the identity from \Cref{lem:laplace_log_g}. This yields 
    \begin{equation}
    \label{eq:Agzeta}
        \begin{aligned}
            \scr{A}^g\zeta_t &= -\log g_t \partial_t\eta(t) -\eta(t)\partial_t \log g_t - \eta(t) \scr{L}^g\log g_t\\
            &= \zeta_t \partial_t \log \eta(t) - \eta(t)\left(\partial_t + V + \nabla\log g + \frac12\Delta\right)\log g_t \\
            &= \zeta_t\partial_t\log\eta(t) - \eta(t)V(\log g)t - \eta(t)\abs{\nabla\log g_t}^2 - \eta(t)\left(\partial_t+\frac12\Delta\right)\log g_t \\
            &= \zeta_t\partial_t\log \eta(t) - \eta(t)V\log g_t -\frac12 \eta(t)\abs{\nabla\log g_t}^2
        \end{aligned}
    \end{equation}
    Upon substituting \eqref{eq:Agzeta} into \eqref{eq:Agxi}, we obtain 
    \begin{equation}
        \label{eq:Agxi2}
        \scr{A}^g\xi_t = \xi_t \partial_t \log\eta(t) + \eta(t)\frac{2d-1}{2(T-t)} - \eta(t)V\log g_t - \frac12 \eta(t)\abs{\nabla\log g_t}^2
    \end{equation}
    We proceed to find an upper bound on $\scr{A}^g \xi_t$. Since $\eta$ and $\abs{\nabla\log g}^2$ are nonnegative on $[t_0,T)$, we immediately have 
        \begin{equation}
        \label{eq:Agxi-firstineq}
        \scr{A}^g\xi_t \leq \xi_t \partial_t \log\eta(t) + \eta(t)\frac{2d-1}{2(T-t)} - \eta(t)V\log g_t 
    \end{equation}
    Similar to the proof that Assumption \ref{ass:main-likelihood} is satisfied in the proof of  \Cref{thm:absolute-continuity-drift}, we find a positive constant $c_2$ such that 
    \begin{equation} 
    \label{eq:Vlogg-bound}
        \abs{V\log g_t} \leq c_2\left( \frac{\dist(X_t,x_T)}{T-t}+\frac{1}{\sqrt{T-t}}\right) \leq c_2\left( \sqrt{\frac{2\xi_t}{\eta(t)(T-t)}} + \frac{1}{\sqrt{T-t}}\right)
    \end{equation}
    Where the second inequality follows again follows from \eqref{eq:zeta-bounds} as $\dist(X_t,x_T) \leq \sqrt{{2(T-t)\xi_t}/{\eta(t)}}$

    Using the preceding, we can find an upper bound for $\xi_t$, as displayed in \eqref{eq:xi-ito}, as follows. We bound $M_t^\xi$ using \eqref{eq:Mtxi-bound}. To bound $\int_{t_0}^t \scr{A}^g\xi_s\dd s$, we use the expression derived in \eqref{eq:Agxi-firstineq}, in which we apply \eqref{eq:Vlogg-bound}. By combining all of this, we obtain 
    \begin{equation}
    \begin{aligned}
        \xi_t &\leq \xi_{t_0} + C\log\left(\frac{1}{T-t}\right) \\
        &\quad +  \frac{c_1^2}{2}\int_{t_0}^t \eta(s)^2\left(\sqrt{\frac{2\xi_s}{\eta(s)(T-s)}} + \frac{1}{\sqrt{T-s}}\right)^2 \dd s \\
        &\quad + \int_{t_0}^t \xi_s \partial_s\log\eta(s) \dd s 
         +\int_{t_0}^t \eta(s)\frac{2d-1}{T-s}\dd s   \\
        &\quad + c_2  \int_{t_0}^t \eta(s) \left(\sqrt{\frac{2\xi_s}{\eta(s)(T-s)}} + \frac{1}{\sqrt{T-s}}\right) \dd s
        \end{aligned}
    \end{equation}
    Summarizing: 
    \begin{equation}
        \xi_t \leq  \xi_{t_0} + C\log\left(\frac{1}{T-t}\right) +\int_{t_0}^t \ell_0(s) \dd s+\int_{t_0}^t \ell_1(s)\xi_s^{1/2} \dd s+\int_{t_0}^t \ell_2(s)\xi_s \dd s , 
    \end{equation}
    where 
    \begin{equation}
        \begin{aligned}
            \ell_0(s) &= \frac{c_1^2}{2}\frac{\eta(s)^2}{T-s} + \eta(s)\frac{2d-1}{T-s} + c_2\frac{\eta(s)}{\sqrt{T-s}}\\
            \ell_1(s) &= \eta(s)^{3/2}\frac{c_1^2\sqrt{2}}{T-s} + c_2\sqrt{\frac{2\eta(s)}{T-s}}\\ 
            \ell_2(s) &= \frac{c_1^2\eta(s)}{T-s} + \partial_s\log\eta(s)
        \end{aligned}
    \end{equation}

    We now apply Theorem 2.1 of \cite{agarwal2005generalization} using, in their notation, $a(t) = \xi(t_0,X_{t_0}) + \ell_0(t)$, $b_1(t)=b_2(t)=t$, $f_1(t,s)=\ell_1(s)$, $f_2(t,s)=\ell_2(s)$, $w_1(u) = u^{1/2}$, $w_2(u) = u$, $u_1=0$ and $u_2=1$. This then yields that 
\[ \xi_t\leq \left(\sqrt{\xi_{t_0} + C\log\left(\frac{1}{T-t}\right) + \int_{t_0}^t \ell_0(s)\dd s}+\frac12\int_0^t \ell_1(s)\dd s\right)^2 \exp\left(\int_{t_0}^t \ell_2(s)\dd s \right) \]

By observation \ref{obs:integrable}, the map $s\mapsto \eta(s)/(T-s)$ is integrable on $[t_0,T)$ and thus $\ell_0$ and $\ell_1$ are integrable as well as linear combinations of compositions that preserve integrability. Hence, $\int_{t_0}^t \ell_0(s)\dd s$ and $\int_{t_0}^t \ell_1(s)\dd s$ are bounded in $\lim_{t\uparrow T}$. Moreover, since
\begin{equation}
\begin{aligned}
    \exp\left(\int_{t_0}^t \ell_2(s)\dd s \right) &= \frac{\eta(t)}{\eta(t_0)} \exp\left(c_1^2 \int_{t_0}^t \frac{\eta(s)}{T-s} \dd s \right),
\end{aligned}
\end{equation}
    and, by observation \ref{obs:logbounded} $\eta(t)\log\left(\frac{1}{T-t}\right)$ is bounded, $\xi_t$ is almost surely bounded. 
\end{proof}

\begin{lem}
\label{lem:laplace_log_g}
    Suppose $g$ satisfies the heat equation $(\partial_t +\frac12\Delta) g=0$ Then 
    \[ \partial_t \log g + \frac12 \Delta \log g = -\frac12 \abs{\log g}^2\]
\end{lem}
\begin{proof}
Using the product rule for the divergence operator
    \begin{equation*}
        \begin{aligned}
            \frac{1}{2}\Delta\log g &= \frac12 \mathrm{div}\left(\frac{\nabla g}{g}\right) = \frac12\inner{\nabla\left(\frac{1}{g}\right)}{\nabla g} + \frac{1}{2g}\mathrm{div}(\nabla g) \\
            &= -\inner{\frac{\nabla g}{g}}{\frac{\nabla g}{g}} - \frac{\partial_t g}{g} = -\abs{\nabla\log g}^2 - \partial_t \log g
        \end{aligned}
    \end{equation*}
\end{proof}

\begin{appendix}
\section{Theorem 4.2 of \cite{palmowski2002}}

	\begin{thm}\label{thm:PR}
		Let $h\in\mathcal{H}$ be such that for all $t\in[0,T)$, $\bb{P}$-almost surely
				\[ \int_0^t \abs{\scr{A}h(s,X_s)}\dd s < \infty \qquad \text{and}\qquad \int_0^t \abs{\frac{\scr{A}h(s,X_s)}{h(s,X_s)}}\dd s < \infty \] 	
		Then $X$ is Markovian under $\bb{P}^h$. Moreover, if there exists a core $\mathbf{D}\subseteq\scr{D}(\scr{A})$ such that for all $f\in\mathbf{D}$, $fh\in\mathbf{D}$, then $X$ is a solution to the martingale problem for $\scr{A}^h$ on $\left(\Omega,\mathcal{F},\{\mathcal{F}_t\}_t,\bb{P}^h\right)$, where $\scr{A}^h$ is given by 
		\begin{equation} 
		\label{eq:conditioned-generator}
			\scr{A}^hf = \frac{\scr{A}(fh)-f\scr{A}h}{h} = \scr{A}f + \frac{1}{h}\Gamma(f,h), \qquad f\in\mathbf{D},
		\end{equation}
		where $\Gamma$ denotes the operateur Carr\'e du Champ $\Gamma(f,h)=\scr{A}(fh)-f\scr{A}h-h\scr{A}f$.
	\end{thm}
    \section{Numerical approximation of the guiding term on the Poincar\'e disk}
    \label{app:heat_kernel_H2}

To evaluate the guiding term of the diffusion bridges, we note that 
\[ \nabla_x \log p(s,x;t,y) = \nabla_x\log\mathbb{E}_{N(0,t-s)}\left[ \frac{Z}{\sqrt{\cosh Z - \cosh\rho_x}}\ind{\{Z>\rho_x\}}\right].\]
 Hence, the guiding term can be estimated using Monte Carlo sampling. Since $\rho_x\ge 0$ this will be very inefficient. To resolve this, we propose importance sampling. Note that with $\tau=\sqrt{t-s}$

\begin{align*}& \nabla_x\log\mathbb{E}_{N(0,\tau^2)}\left[ \frac{Z}{\sqrt{\cosh Z - \cosh\rho_x}}\ind_{\{Z>\rho_x\}}\right] = \nabla_x\log \bb{E}_{N(\mu,\tau^2)}\left[ \frac{Z\exp\left(-\frac{\mu Z}{\tau^2}+\frac{\mu^2}{2\tau^2}\right)}{\sqrt{\cosh Z-\cosh\rho_x}}\ind\{Z>\rho_x\}\right] 
\\ & \qquad\qquad = \nabla_x\log \bb{E}_{N(0,1)}\left[ \frac{(\mu+\tau Z)\exp\left(-\frac{\mu Z}{\tau}-\frac{\mu^2}{2\tau^2}\right)}{\sqrt{\cosh (\mu+\tau Z)-\cosh\rho_x}}\ind\{\mu+\tau Z>\rho_x\}\right] 
\\ &\qquad \qquad  = -\frac1{2\tau^2} \nabla_x \mu^2 +  \nabla_x\log \bb{E}_{N(0,1)}\left[ \frac{(\mu+\tau Z)\exp\left(-\frac{\mu Z}{\tau}\right)}{\sqrt{\cosh (\mu+\tau Z)-\cosh\rho_x}}\ind\{\mu+\tau Z>\rho_x\}\right]. 
\end{align*}
To ensure that at least $97\%$ of the samples exceed $\rho_x$, we choose $\mu\equiv\mu(x,\tau) = \rho_x + 2\tau$. This implies that the event in the indicator in the last line of the preceding display equals $\{Z>-2 \}$. We approximate the term with the expectation by Monte-Carlo simulation. Let $\{Z_i\}$ be independent $N(0,1)$-distributed random variables and define
\[ w(x,\tau, z) = \frac{(\mu(x,\tau) + \tau z) \exp(-\mu(x,\tau) z/\tau)}{\sqrt{\cosh(\mu(x,\tau)+\tau z)-\cosh(\rho_x)}} \mathbf{1}\{z>-2\}.   \]
Then we approximate $\nabla_x \log p(s,x;t,y)$ by 
\[ -\frac1{2\tau^2} \nabla_x \mu(x,\tau)^2 +\frac{ \nabla_x \sum_i   w(x,\tau, Z_i)}{\sum_i w(x,\tau, Z_i)},  \]
which is evaluated by (forward) automatic differentiation. When evaluating this gradient for different values of $x$, we fix the realisations of the random variables $\{Z_i\}$, where we can simply discard those realisations that are below $-2$.

\end{appendix}
\section*{funding}
This work is part of the research project ‘Bayes for longitudinal data on manifolds’ with project number OCENW.KLEIN.218, which is financed by the Nederlandse Organisatie voor Wetenschappelijk Onderzoek (NWO). The fourth author was supported by a VILLUM FONDEN grant (VIL40582), and the Novo Nordisk Foundation grant NNF18OC0052000.

\bibliographystyle{authordate1} 
\bibliography{biblio.bib}   

\begin{thebibliography}{}

\bibitem[\protect\citename{Agarwal {\em et~al.}, }2005]{agarwal2005generalization}
Agarwal, Ravi, Deng, Shengfu, \& Zhang, Weinian. 2005.
\newblock Generalization of a retard Gronwall-like inequality and its applications.
\newblock {\em Applied Mathematics and Computation}, {\bf 165}(06), 599--612.

\bibitem[\protect\citename{Arnaudon {\em et~al.}, }2022]{arnaudonDiffusionBridgesStochastic2022}
Arnaudon, Alexis, van~der Meulen, Frank, Schauer, Moritz, \& Sommer, Stefan. 2022.
\newblock Diffusion {Bridges} for {Stochastic} {Hamiltonian} {Systems} and {Shape} {Evolutions}.
\newblock {\em SIAM Journal on Imaging Sciences}, {\bf 15}(1), 293--323.

\bibitem[\protect\citename{Axen {\em et~al.}, }2021]{manifolds}
Axen, Seth~D., Baran, Mateusz, Bergmann, Ronny, \& Rzecki, Krzysztof. 2021.
\newblock Manifolds.jl: An Extensible Julia Framework for Data Analysis on Manifolds.
\newblock {\em arXiv:2106.08777}.

\bibitem[\protect\citename{Baudoin, }2002]{baudoin2002}
Baudoin, Fabrice. 2002.
\newblock Conditioned stochastic differential equations: theory, examples and application to finance.
\newblock {\em Stochastic Processes and their Applications}, {\bf 100}(1), 109--145.

\bibitem[\protect\citename{Baudoin, }2004]{baudoin2004conditioning}
Baudoin, Fabrice. 2004.
\newblock Conditioning and initial enlargement of filtration on a Riemannian manifold.
\newblock {\em Annals of probability},  2286--2303.

\bibitem[\protect\citename{Beskos {\em et~al.}, }2006a]{BeskosPapaspiliopoulosRobertsFearnhead}
Beskos, Alexandros, Papaspiliopoulos, Omiros, Roberts, Gareth~O., \& Fearnhead, Paul. 2006a.
\newblock Exact and computationally efficient likelihood-based estimation for discretely observed diffusion processes.
\newblock {\em J. R. Stat. Soc. Ser. B Stat. Methodol.}, {\bf 68}(3), 333--382.
\newblock With discussions and a reply by the authors.

\bibitem[\protect\citename{Beskos {\em et~al.}, }2006b]{BeskosPapaspiliopoulosRoberts}
Beskos, Alexandros, Papaspiliopoulos, Omiros, \& Roberts, Gareth~O. 2006b.
\newblock Retrospective exact simulation of diffusion sample paths with applications.
\newblock {\em Bernoulli}, {\bf 12}(6), 1077--1098.

\bibitem[\protect\citename{Beskos {\em et~al.}, }2008]{beskos-mcmc-methods}
Beskos, Alexandros, Roberts, Gareth, Stuart, Andrew, \& Voss, Jochen. 2008.
\newblock MCMC Methods for diffusion bridges.
\newblock {\em Stochastics and Dynamics}, {\bf 08}(03), 319--350.

\bibitem[\protect\citename{Bierkens {\em et~al.}, }2020]{bierkens2020simulation}
Bierkens, Joris, Van Der~Meulen, Frank, \& Schauer, Moritz. 2020.
\newblock Simulation of elliptic and hypo-elliptic conditional diffusions.
\newblock {\em Advances in Applied Probability}, {\bf 52}(1), 173--212.

\bibitem[\protect\citename{Bierkens {\em et~al.}, }2021]{bierkens2020piecewise}
Bierkens, Joris, Grazzi, Sebastiano, van~der Meulen, Frank, \& Schauer, Moritz. 2021.
\newblock {\em A piecewise deterministic Monte Carlo method for diffusion bridges}.

\bibitem[\protect\citename{Bladt {\em et~al.}, }2016]{bladt2016}
Bladt, Mogens, Finch, Samuel, \& S\o{}rensen, Michael. 2016.
\newblock Simulation of multivariate diffusion bridges.
\newblock {\em J. R. Stat. Soc. Ser. B. Stat. Methodol.}, {\bf 78}(2), 343--369.

\bibitem[\protect\citename{Bui {\em et~al.}, }2023]{bui2023inference}
Bui, Mai~Ngoc, Pokern, Yvo, \& Dellaportas, Petros. 2023.
\newblock Inference for partially observed Riemannian Ornstein--Uhlenbeck diffusions of covariance matrices.
\newblock {\em Bernoulli}, {\bf 29}(4), 2961--2986.

\bibitem[\protect\citename{Clark, }1990]{clark1990}
Clark, J. M.~C. 1990.
\newblock The simulation of pinned diffusions.
\newblock {\em Pages  1418--1420 of:} {\em Decision and Control, 1990., Proceedings of the 29th IEEE Conference on}.
\newblock IEEE.

\bibitem[\protect\citename{Corstanje, }2024]{manifold_bridge}
Corstanje, Marc. 2024.
\newblock \url{https://github.com/macorstanje/manifold_bridge.jl}.
\newblock Code examples for simulations.

\bibitem[\protect\citename{Corstanje {\em et~al.}, }2023]{corstanje2021conditioning}
Corstanje, Marc, van~der Meulen, Frank, \& Schauer, Moritz. 2023.
\newblock Conditioning continuous-time Markov processes by guiding.
\newblock {\em Stochastics}, {\bf 95}(6), 963--996.

\bibitem[\protect\citename{Delyon \& Hu, }2006]{delyon2006}
Delyon, B., \& Hu, Y. 2006.
\newblock Simulation of conditioned diffusions and application to parameter estimation.
\newblock {\em Stochastic Processes and their Applications}, {\bf 116}(11), 1660--1675.

\bibitem[\protect\citename{Durham \& Gallant, }2002]{durham2002}
Durham, Garland~B., \& Gallant, A.~Ronald. 2002.
\newblock Numerical techniques for maximum likelihood estimation of continuous-time diffusion processes.
\newblock {\em J. Bus. Econom. Statist.}, {\bf 20}(3), 297--338.
\newblock With comments and a reply by the authors.

\bibitem[\protect\citename{Elworthy, }1988]{elworthy1988geometric}
Elworthy, David. 1988.
\newblock Geometric aspects of diffusions on manifolds.
\newblock {\em Pages  277--425 of:} Hennequin, Paul-Louis (ed), {\em {\'E}cole d'{\'E}t{\'e} de Probabilit{\'e}s de Saint-Flour XV--XVII, 1985--87}.
\newblock Berlin, Heidelberg: Springer Berlin Heidelberg.

\bibitem[\protect\citename{Emery, }1989]{emeryStochasticCalculusManifolds1989}
Emery, Michel. 1989.
\newblock {\em Stochastic {Calculus} in {Manifolds}}.
\newblock Universitext.
\newblock Berlin, Heidelberg: Springer Berlin Heidelberg.

\bibitem[\protect\citename{Golightly \& Wilkinson, }2010]{GolightlyWilkinsonChapter}
Golightly, Andrew, \& Wilkinson, Darren~J. 2010.
\newblock {\em Learning and Inference in Computational Systems Biology}.
\newblock MIT Press.
\newblock Chap. {Markov chain Monte Carlo algorithms for SDE parameter estimation}, pages  253--276.

\bibitem[\protect\citename{Hairer {\em et~al.}, }2009]{hairer2009}
Hairer, Martin, Stuart, Andrew~M., \& Voss, Jochen. 2009.
\newblock Sampling Conditioned Diffusions.
\newblock {\em Pages  159--186 of:} {\em Trends in Stochastic Analysis}.
\newblock London Mathematical Society Lecture Note Series, vol. 353.
\newblock Cambridge University Press.

\bibitem[\protect\citename{Hansen {\em et~al.}, }2021]{hansen2021geometric}
Hansen, Pernille, Eltzner, Benjamin, \& Sommer, Stefan. 2021.
\newblock Diffusion Means and Heat Kernel on Manifolds.
\newblock {\em Pages  111--118 of:} Nielsen, Frank, \& Barbaresco, Fr{\'e}d{\'e}ric (eds), {\em Geometric Science of Information}.
\newblock Cham: Springer International Publishing.

\bibitem[\protect\citename{Hsu, }2002]{Hsu2002}
Hsu, Elton~P. 2002.
\newblock {\em Stochastic analysis on manifolds}.
\newblock  Vol. 38.
\newblock American Mathematical Soc.

\bibitem[\protect\citename{Jensen \& Sommer, }2021]{jensen2021simulation}
Jensen, Mathias~H{\o}jgaard, \& Sommer, Stefan. 2021.
\newblock Simulation of Conditioned Semimartingales on Riemannian Manifolds.
\newblock {\em arXiv preprint arXiv:2105.13190}.

\bibitem[\protect\citename{Jensen {\em et~al.}, }2019]{jensen2019simulation}
Jensen, Mathias~H{\o}jgaard, Mallasto, Anton, \& Sommer, Stefan. 2019.
\newblock Simulation of conditioned diffusions on the flat torus.
\newblock {\em Pages  685--694 of:} {\em Geometric Science of Information: 4th International Conference, GSI 2019, Toulouse, France, August 27--29, 2019, Proceedings 4}.
\newblock Springer.

\bibitem[\protect\citename{Jensen {\em et~al.}, }2022]{jensen2022discrete}
Jensen, Mathias~H{\o}jgaard, Joshi, Sarang, \& Sommer, Stefan. 2022.
\newblock Discrete-Time Observations of Brownian Motion on Lie Groups and Homogeneous Spaces: Sampling and Metric Estimation.
\newblock {\em Algorithms}, {\bf 15}(8), 290.

\bibitem[\protect\citename{K{\"u}hnel \& Sommer, }2017]{kuhnel2017computational}
K{\"u}hnel, Line, \& Sommer, Stefan. 2017.
\newblock Computational Anatomy in Theano.
\newblock {\em Pages  164--176 of:} {\em Graphs in Biomedical Image Analysis, Computational Anatomy and Imaging Genetics}.
\newblock Cham: Springer International Publishing.

\bibitem[\protect\citename{Lin {\em et~al.}, }2010]{LinChenMykland}
Lin, Ming, Chen, Rong, \& Mykland, Per. 2010.
\newblock On generating {M}onte {C}arlo samples of continuous diffusion bridges.
\newblock {\em J. Amer. Statist. Assoc.}, {\bf 105}(490), 820--838.

\bibitem[\protect\citename{Lovett, }2010]{Lovett2010}
Lovett, Stephen. 2010.
\newblock {\em Differential geometry of manifolds}.
\newblock AK Peters/CRC Press.

\bibitem[\protect\citename{Manton, }2013]{manton2013primer}
Manton, Jonathan~H. 2013.
\newblock A primer on stochastic differential geometry for signal processing.
\newblock {\em IEEE Journal of Selected Topics in Signal Processing}, {\bf 7}(4), 681--699.

\bibitem[\protect\citename{Mider {\em et~al.}, }2021]{mider2021continuous}
Mider, Marcin, Schauer, Moritz, \& van~der Meulen, Frank. 2021.
\newblock {Continuous-discrete smoothing of diffusions}.
\newblock {\em Electronic Journal of Statistics}, {\bf 15}(2), 4295 -- 4342.

\bibitem[\protect\citename{Palmowski \& Rolski, }2002]{palmowski2002}
Palmowski, Zbigniew, \& Rolski, Tomasz. 2002.
\newblock A technique for exponential change of measure for Markov processes.
\newblock {\em Bernoulli}, {\bf 8}(6), 767--785.

\bibitem[\protect\citename{Papaspiliopoulos {\em et~al.}, }2013]{PapaRobertsStramer}
Papaspiliopoulos, Omiros, Roberts, Gareth~O., \& Stramer, Osnat. 2013.
\newblock Data {A}ugmentation for {D}iffusions.
\newblock {\em J. Comput. Graph. Statist.}, {\bf 22}(3), 665--688.

\bibitem[\protect\citename{Roberts \& Rosenthal, }2009]{roberts2009examples}
Roberts, Gareth~O, \& Rosenthal, Jeffrey~S. 2009.
\newblock Examples of adaptive MCMC.
\newblock {\em Journal of computational and graphical statistics}, {\bf 18}(2), 349--367.

\bibitem[\protect\citename{Schauer {\em et~al.}, }2017]{schauer_guidedproposals2017}
Schauer, Moritz, van~der Meulen, Frank, \& van Zanten, Harry. 2017.
\newblock Guided proposals for simulating multi-dimensional diffusion bridges.
\newblock {\em Bernoulli}, {\bf 23}(4A), 2917--2950.

\bibitem[\protect\citename{Sommer, }2020]{sommerProbabilisticApproachesGeometric2020}
Sommer, Stefan. 2020.
\newblock Probabilistic approaches to geometric statistics: {Stochastic} processes, transition distributions, and fiber bundle geometry.
\newblock {\em Pages  377--416 of:} Pennec, Xavier, Sommer, Stefan, \& Fletcher, Tom (eds), {\em Riemannian {Geometric} {Statistics} in {Medical} {Image} {Analysis}}.
\newblock Academic Press.

\bibitem[\protect\citename{Sommer {\em et~al.}, }2017]{sommerBridgeSimulationMetric2017}
Sommer, Stefan, Arnaudon, Alexis, Kuhnel, Line, \& Joshi, Sarang. 2017.
\newblock Bridge {Simulation} and {Metric} {Estimation} on {Landmark} {Manifolds}.
\newblock {\em Pages  79--91 of:} {\em Graphs in {Biomedical} {Image} {Analysis}, {Computational} {Anatomy} and {Imaging} {Genetics}}.
\newblock Lecture {Notes} in {Computer} {Science}.
\newblock Springer.

\bibitem[\protect\citename{Stuart {\em et~al.}, }2004]{Stuart}
Stuart, Andrew~M., Voss, Jochen, \& Wiberg, Petter. 2004.
\newblock Fast communication conditional path sampling of {SDE}s and the {L}angevin {MCMC} method.
\newblock {\em Commun. Math. Sci.}, {\bf 2}(4), 685--697.

\bibitem[\protect\citename{Sturm, }1993]{sturm1993}
Sturm, K.T. 1993.
\newblock Schr{\"o}dinger Semigroups on Manifolds.
\newblock {\em Journal of Functional Analysis}, {\bf 118}(2), 309--350.

\bibitem[\protect\citename{Whitaker {\em et~al.}, }2017]{whitaker2017}
Whitaker, Gavin~A., Golightly, Andrew, Boys, Richard~J., \& Sherlock, Chris. 2017.
\newblock Improved bridge constructs for stochastic differential equations.
\newblock {\em Statistics and Computing}, {\bf 27}(4), 885--900.

\end{thebibliography}
\end{document}